\documentclass[11pt]{article}
\usepackage{amscd}
\usepackage{amsmath}
\usepackage{latexsym}
\usepackage{amsfonts}
\usepackage{amssymb}
\usepackage{amsthm}
\usepackage{graphicx}
\usepackage{verbatim}
\usepackage{mathrsfs}
\usepackage{enumerate}
\usepackage{hyperref}
\usepackage{fullpage}
\usepackage{color}

\theoremstyle{plain}\newtheorem{definition}{Definition}[section]
\theoremstyle{definition}\newtheorem{theorem}{Theorem}[section]
\theoremstyle{plain}\newtheorem{lemma}[theorem]{Lemma}
\theoremstyle{plain}
\theoremstyle{plain}\newtheorem{proposition}[theorem]{Proposition}
\theoremstyle{remark}\newtheorem{remark}{Remark}[section]

\newcommand{\norm}[1]{\left\|#1\right\|}

\newcommand{\NN}{\mathbb{N}}

\newcommand{\RR}{\mathbb{R}}

\allowdisplaybreaks
\begin{document}
\title{On the global well-posedness for the Boussinesq system with horizontal dissipation}
\author{Changxing Miao$^1$ and  Xiaoxin Zheng$^2$\\
\\
        \small{$^{1}$ Institute of Applied Physics and Computational Mathematics,}\\
        \small{P.O. Box 8009, Beijing 100088, P.R. China.}\\
        \small{(miao\_{}changxing@iapcm.ac.cn)}\\
        \small{$^2$  The Graduate School of China Academy of Engineering Physics,}\\
        {\small P.O. Box 2101, Beijing 100088, P.R. China.}\\
        {\small{(xiaoxinyeah@163.com)}}}

\date{}
\maketitle

\begin{abstract}
In this paper, we investigate the Cauchy problem for the
tridimensional Boussinesq equations with horizontal dissipation.
Under the assumption that the initial data is an axisymmetric
without swirl, we prove the global well-posedness for this system.
In the absence of vertical dissipation, there is no smoothing effect
on the vertical derivatives. To make up this shortcoming, we first
establish a magic relationship between $\frac{u^{r}}{r}$ and
$\frac{\omega_\theta}{r}$ by taking full advantage of the structure
 of the axisymmetric fluid without swirl and some tricks in harmonic analysis. This together with the structure of
 the coupling of \eqref{eq1.1} entails the desired regularity.
\end{abstract}


\section{Introduction}
The Boussinesq system describes the influence of the convection phenomenon in the dynamics of the ocean or atmosphere.
In fact, it is used as a toy model for geophysical flows whenever rotation and stratification play an important role (see \cite{J-P}). This system is described  by the following equations:
\begin{equation}\label{full}
    \begin{cases}
       (\partial_{t}+u\cdot\nabla)u-\kappa\Delta u+\nabla p=\rho e_{n},\quad(t,x)\in\mathbb{R}_{+}\times\mathbb{R}^{n},\quad n=2,3,\\
       (\partial_{t}+u\cdot\nabla)\rho-\nu\Delta\rho=0,\\
       \text{div}u=0,\\
       (u,\rho)|_{t=0}=(u_{0},\rho_{0}),
    \end{cases}
\end{equation}
where, the velocity  $u=(u^1,\cdots,u^n)$ is a vector field with zero divergence and $\rho$ is a scalar quantity such as the concentration of a chemical substance or the temperature variation in a gravity fields, in which case $\rho e_n$ represents the buoyancy force. The nonnegative parameters $\kappa$ and $\nu$ denote the viscosity and the molecular diffusion respectively. In addition, the pressure $p$ is a scalar quantity which can be expressed by the unknowns $u$ and $\rho$.

In the case where $\nu$ and $\kappa$ are nonnegative constants, the local well-posedness of \eqref{full} can be easily established by using the energy method.
When variables $\kappa$ and $\nu$ are both positive, the classical methods allow to establish the global existence of regular solutions in dimension two and for three dimension with small initial data. Unfortunately, for the inviscid Boussinesq \mbox{system \eqref{full}}, whether or not smooth solution for some nonconstant $\rho_0$ blows up in finite time is still an open problem. The intermediate situation has been attracted considerable attentions in the past years and important progress has been made.  When $\nu$ is a positive constant and  $\kappa=0$; or $\nu=0$ and $\kappa$ is a positive constant, D. Chae \cite{ha}, and T.Y. Hou and C. Li \cite{hou-li} proved the global well-posedness independently for the two-dimensional Boussinesq system. It is also shown the global well-posedness in the critical spaces, see \cite{ah}. In addition, C. Miao and L. Xue \cite{CMX} proved the global well-posedness of the
two-dimensional Boussinesq equations with fractional viscosity and thermal diffusion when the fractional powers obey mild condition.   Other interesting results on the two-dimensional Boussinesq equations can be found in \cite{acw,acw1,HKR1,HKR2}.

Recently, there are many works devoted to the study of the tridimensional axisymmetric Boussinesq system without swirl for
different viscosities. In \cite{A-H-K0}, a global result was established but under some restrictive conditions on the initial density, namely it does not intersect the axis $r=0$. Subsequently, T. Hmidi and F. Rousset \cite{hrou1} removed the assumption on the support of the density and proved the global well-posedness for the Navier-Stokes-Boussinesq system by virtue of the structure of the coupling between two equations of \eqref{full} with $\nu=0$. In \cite{hrou}, they also  proved the global well-posedness for the tridimensional Euler-Boussinesq system with axisymmetric initial data without swirl.

In the present paper, we consider the case that the diffusion and the viscosity only occur in the horizontal direction. More precisely,
\begin{equation}\label{eq1.1}
    \begin{cases}
       (\partial_{t}+u\cdot\nabla)u-\Delta_{h}u+\nabla p=\rho e_{z},\quad(t,x)\in\mathbb{R}_{+}\times\mathbb{R}^{3},\\
       (\partial_{t}+u\cdot\nabla)\rho-\Delta_{h}\rho=0,\\
       \text{div}u=0,\\
       (u,\rho)|_{t=0}=(u_{0},\rho_{0}),
    \end{cases}
\end{equation}
Here $\Delta_{h}=\partial^{2}_{1}+\partial_{2}^{2}$.  Let us point out that the anisotropic dissipation assumption is natural in the studying of geophysical fluids. It turns out that, in certain regimes and after suitable rescaling, the vertical dissipation (or the horizontal dissipation) is negligible  as compared to the horizontal dissipation (or the vertical dissipation)\mbox{(see \cite{Cdgg} for details)}. In fact, there are several works devoted to study of the two-dimensional Boussinesq system with anisotropic dissipation. In \cite{dp2}, R. Danchin and M. Paicu  proved the global existence for the two-dimensional Boussinesq system with horizontal viscosity in only one equation. They mainly exhibited a polynomial control of $\norm{\nabla u}_{\sqrt{L}}$, where the space $
\sqrt{L}$ stands for the space of functions $f$ in $\cap_{2\leq p<\infty}L^p$ such that
\begin{equation}
\|f\|_{\sqrt{L}}:=\sup_{2\leq p<\infty}{p}^{-\frac{1}{2}}\|f\|_{L^{p}}\leq \infty.
 \end{equation}
 Combining this with the following estimate
$$\norm{\nabla u}_\infty\leq C\big(1+\norm{\nabla u}_{\sqrt{L}}\log(e+\norm{u}_{H^{s}})\big),\quad s>2
$$ yields the global well-posedness of smooth solutions. Next, they observed the fact $\norm{\nabla u}_{\sqrt{L}}$ implies that $u\in L^{2}_{\rm loc}(\RR^+,{\rm LogLip^{\frac12}})$, where ${\rm LogLip^{\frac12}}$ stands for the set of bounded functions $f$ such that
\begin{equation}\sup_{x\neq y;|x-y|\leq\frac12}\frac{|f(y)-f(x)|}{|x-y|\log^{\frac12}(|x-y|)^{-1}}\leq +\infty.
\end{equation} And then they established the global existence with uniqueness for rough data with the help of a losing estimate. Recently,  A. Adhikari, C. Cao and J. Wu also established some global results for different model under various assumption on dissipation in a series of recent papers, see in particular \cite{acw,acw1,cwMhd,cw}. In \cite{cw}, \mbox{C. Cao and J. Wu } proved the global well-posedness for the two-dimensional Boussinesq system with vertical viscosity and vertical diffusion  in terms of a Log-type inequality. In their proof, they first find that ${L^{p}}$-norm on vertical component of velocity with $2\leq p<\infty$ at any time does not grow faster than $\sqrt{p~ {\rm\log}p}$ as $p$ increase by means of the low-high decomposition techniques.

To better understand the axisymmetric fields, let us recall some algebraic and geometric properties of the axisymmetric vector fields and discuss the special structure of the vorticity of system  \eqref{eq1.1}, \mbox{see for example \cite{Hmidi,CM-1}}.
First, we give some general statement in cylindrical coordinates: we say that a vector field $u$ is axisymmetric if it
satisfies
\begin{equation}\label{axisymmetric-def}
\mathcal{R}_{-\alpha}\{u(\mathcal{R}_{\alpha}x)\}=u(x),\quad\forall\alpha \in[0,2\pi],\quad \forall x\in \mathbb{R}^{3},
\end{equation}
where $\mathcal{R}_{\alpha}$ denotes the rotation of axis $(Oz)$ and with angle $\alpha$. Moreover, an axisymmetric vector field $u$ is called without swirl if it has the form:
$$u(t,x)=u^{r}(r,z)e_{r}+u^{z}(r,z)e_{z},\quad x=(x_1,x_2,x_3),\quad r=\sqrt{x^{2}_{1}+x^{2}_{2}}\quad\text{and }z=x_{3},
$$
where $(e_{r},e_{\theta},e_{z})$ is the cylindrical basis of $\mathbb{R}^{3}$. Similarly,
a scalar function $f:\RR^3\to\RR$ is called axisymmetric if the vector field $x\mapsto f(x)e_z$ is axisymmetric, which means that
\begin{equation}\label{scalar}
f(\mathcal{R}_\alpha x)=f(x),\quad\,\forall\, x\in\RR^3, \quad\forall\,\alpha\in[0,2\pi].
\end{equation}
This is equivalent to say that $f$ depends only on $r$ and $z$.
Direct computations show us that the vorticity $\omega:=\text{curl} u$ of the vector
field $u$ takes the form
$$\omega=(\partial_{z}u^{r}-\partial_{r}u^{z})e_\theta:=\omega_\theta e_\theta.
$$
On the other hand, we know that
 \begin{equation}\label{axisy-1}
u\cdot\nabla=u^{r}\partial_{r}+u^{z}\partial_{z},\quad \text{div}u=\partial_{r}u^{r}+\frac{u^{r}}{r}+\partial_{z}u^{z}\quad \text{and}\quad\omega\cdot \nabla u=\frac{u^{r}}{r}\omega
\end{equation}
in the cylindrical coordinates. Therefore, the vorticity $\omega$ satisfies
\begin{equation}
 \label{tourbillon-0}
\partial_t \omega +u\cdot\nabla\omega-\Delta_{h}\omega
 =-\partial_{r}\rho e_{\theta}+\frac{u^r}{r}\omega.
\end{equation}
Since the horizontal Laplacian operator has the form $\Delta_{h}=\partial_{rr}+\frac{1}{r}\partial_{r}$ in the cylindrical coordinates then the $\omega_
\theta$ satisfies
\begin{equation}
 \label{tourbillon}
\partial_t \omega_\theta +u\cdot\nabla\omega_\theta-\Delta_{h}\omega_\theta
+\frac{\omega_\theta}{r^2} =-\partial_{r}\rho+\frac{u^r}{r}\omega_\theta.
\end{equation}
In this paper, we are going to establish the global well-posedness for the system \eqref{eq1.1} corresponding to large axisymmetric data without swirl. Since the dissipation only occurs in the horizontal direction,
it seems not obvious  to get the global regularity of solutions following from \cite{A-H-K0} directly. Indeed, their proof relies on the smoothing effect on vertical direction. Also, we do not expect to obtain the growth estimate of $L^{p}$-norm about vertical component of velocity as in \cite{cw} for the tridimensional axisymmetric Boussinesq equations.
Besides,  as the space $H^1(\RR^3)$ fails to be embedded in $\sqrt{L}(\RR^3)$, it is impossible to obtain the bound of $\norm{\nabla u}_{\sqrt{L}}$
in terms of $\norm{\omega}_{\sqrt{L}}$ just as in \cite{dp2}. This requires us to further study the structure of axisymmetric flows and establish priori estimate
to control the vorticity in $L^{1}_{\text{loc}}(\RR^{+},L^{\infty})$. 
Now, let us briefly to sketch the proof of results.
According to \eqref{tourbillon} and the properties of axisymmetric flows, we find that the quantity $\frac{\omega_\theta}{r}$ satisfies
\begin{equation}\label{ww}
\big(\partial_t+u\cdot\nabla\big)\frac{\omega_\theta}{r}-\big(\Delta_{h}+{{2 \over r}}\partial_r\big) \frac{\omega_\theta}{r} =-\frac{\partial_r\rho}{r}.
\end{equation}
We observe that the main difficulty is the lack of information about the influence of the term in the right side of \eqref{ww} and how to use some priori estimates on $\rho$ to control it.
 Therefore we need to study the properties of the operator $\frac{\partial_r}{r}$ so as to
analyze the influence of the forcing term $\frac{\partial_r\rho}{r}$ on the motion of the fluid. Indeed, the behavior of $\Delta_h+\frac{\partial_r}{r}$ is like that of $\Delta_h$, which be derived from the fact that $\frac{\partial_r}{r}$ is a part of the operator $\Delta_h=\partial^2_r+\frac{\partial_r}{r}$.
 This induces us to consider the structure of the coupling between two equation of \eqref{eq1.1}. From this observation,  we introduce a new quantity $\Gamma:=  \frac{\omega_\theta}{r}  -\frac{1}{2}\rho$  and then $\Gamma$
solves the following transport equation
 \begin{equation}
 (\partial_{t}+u\cdot\nabla)\Gamma-(\Delta_{h}+\frac{2}{r}\partial_{r})\Gamma=0.
 \end{equation}
It follows that
\begin{equation*}
\norm{\Gamma(t)}_{L^p}\leq\norm{\Gamma_0}_{L^p},\quad\forall p\in[1,\infty].
\end{equation*}
This together with the $L^{p}$-estimate of $\rho$ gives that
\begin{equation*}
\Big\|\frac{\omega_\theta}{r}(t)\Big\|_{L^p}\leq\Big\|\frac{\omega_\theta}{r}(0)\Big\|_{L^p},\quad\forall p\in[1,\infty].
\end{equation*}
This estimate enables us to establish a global $H^1$-bound of the velocity. Now, by taking the $L^{2}$-inner product of \eqref{tourbillon} with $\omega_\theta$ and using the anisotropic inequality which will be described in Appendix \ref{appendix}, we obtain
 \begin{equation}\label{H1}
 \begin{split}
&\frac12\frac{\rm d}{{\rm d}t}\|\omega_\theta(t)\|_{L^2}^2+\|\nabla_{h}\omega_\theta(t)\|_{L^2}^2+\Big\|\frac{\omega_\theta}{r}(t)\Big\|_{L^2}^2\\
\leq&\Big\|\frac{u^{r}}{r}\Big\|^{\frac{3}{4}}_{L^{6}}\Big\|\partial_{z}\Big(\frac{u^{r}}{r}\Big)\Big\|^{\frac{1}{4}}_{L^{2}}
\norm{\omega_\theta}^{\frac{1}{2}}_{2}\norm{\nabla_{h}\omega_\theta}^{\frac{1}{2}}_{2}\norm{\omega_\theta}_{2}+\norm{\rho}^{2}_{L^{2}}
+\frac14\|\nabla_{h}\omega_\theta(t)\|_{L^2}^2+\frac14\left\|\frac{\omega_\theta}{r}(t)\right\|_{L^2}^2.
\end{split}
\end{equation}
As a consequence, it is impossible to use the information of $\frac{\omega_\theta}{r}$ to control the quantity $\norm{\partial_z(u^r/r)}_{L^{2}}$  via the following pointwise estimate established by T. Shirota and  T. Yanagisawa (abbr. S-Y)
\begin{equation*}\Big|\frac{u^{r}}{r}\Big|\leq C\frac{1}{|x|^{2}}\ast\Big|\frac{\omega_\theta}{r}\Big|.
\end{equation*}
 This forces us to establish the new relationship between $\frac{u^{r}}{r}$ and $\frac{\omega_\theta}{r}$ instead of the S-Y estimate. To fulfill the goal, we find the following algebraic identity deduced from the geometric structure of axisymmetric flows and the Biot-Savart law:
\begin{equation}
\frac{u^{r}}{r}=\partial_{z}\Delta^{-1}\left(\frac{\omega_\theta}{r}\right)-2\frac{\partial_r}{r}\Delta^{-1}\partial_z\Delta^{-1}\left(\frac{\omega_\theta}{r}\right)\cdot
\end{equation}
This identity allows us to conclude  the S-Y estimate, one can see Propositions \ref{prop-identity} and \ref{prop-sy} for more details.

Before stating our results, let us introduce the space $L$ of those functions $f$ which belong to every space $L^p$ with $2\leq p<\infty$ and  satisfy
\begin{equation}
\|f\|_{L}:=\sup_{2\leq p<\infty}{p}^{-1}\|f\|_{L^{p}}\leq \infty.
 \end{equation}Our results are stated as follows.
\begin{theorem}\label{thm1}
 Let $u_0\in H^1$ be an axisymmetric divergence free  vector field without swirl
  such that $\frac{{\omega_0}}{r}\in L^2$ and $\partial_{z}\omega_{0}\in L^{2} $.
Let $\rho_0\in H^{0,1}$ be an axisymmetric function.
 Then there is  a unique global  solution $(u,\rho)$ of the system \eqref{eq1.1} such that
$$
 u\in\mathcal{C}(\mathbb{R}_+;H^1)\cap L^2_{\textnormal{loc}}(\mathbb{R}_+;H^{1,2}\cap H^{2,1}),\quad \partial_z\omega\in \mathcal{C}(\RR_+;L^{2})\cap L^2_{\textnormal{loc}}(\mathbb{R}_+;H^{1,0}),$$
 $$ \frac{\omega}{r}\in L^\infty_{\textnormal{loc}}(\mathbb{R}_+;L^2)\cap L^2_{\textnormal{loc}}(\mathbb{R}_+;H^{1,0}),\quad
 \rho\in \mathcal{C}(\mathbb{R}_+; H^{0,1})\cap L^2_{\textnormal{loc}}(\mathbb{R}_+;H^{1,1}).
 $$
 Here and in what follows, we can refer to Section \ref{space} for the definition of spaces such as $H^{1}$, $H^{0,1}$, etc.
\end{theorem}
\begin{remark}
The main difficulty is how to establish the $H^1$-estimates of velocity due to the lack of dissipation in the vertical direction.
To overcome this difficulty,
we explore  an  algebraic
identity between $\frac{u^{r}}{r}$ and $\frac{\omega_\theta}{r}$,
which strongly rely on  the geometric structure of axisymmetric
flows,  and control the stretching term in vorticity equation
\begin{equation*}
\partial_t \omega +u\cdot\nabla\omega-\Delta_{h}\omega
 =-\partial_{r}\rho e_{\theta}+\frac{u^r}{r}\omega.
\end{equation*}
We observe the diffusion  in a direction perpendicular to the buoyancy
force, and this  helps  us to control the source term $\partial_{r}\rho e_{\theta}$
 by virtue of the horizontal
smoothing effect.
\end{remark}
\begin{theorem}\label{lose-global}
Let $u_0\in H^{1}$ be an  axisymmetric divergence free  vector field without swirl such that $\frac{{\omega_0}}{r}\in L^2$ and $\omega_{0}\in L^\infty$.
Let $\rho_{0}\in H^{0,1}$ be an axisymmetric function. Then the system \eqref{eq1.1} admits a unique global solution $(\rho,u)$ such that
$$ u\in\mathcal{C}_{w}(\mathbb{R}_+;H^1)\cap L^2_{\textnormal{loc}}(\mathbb{R}_+;H^{1,2}\cap H^{2,1}),\quad \nabla u\in L_{\rm loc}^{\infty}(\RR_{+};{L}),$$
$$\frac{\omega}{r}\in L^\infty_{\textnormal{loc}}(\mathbb{R}_+;L^2)\cap L^2_{\textnormal{loc}}(\mathbb{R}_+;H^{1,0}),\quad\rho\in\mathcal{C}_{w}(\mathbb{R}_{+};H^{0,1})\\
\cap
L_{\textnormal{loc}}^{2}(\RR_{+};H^{1,1})\cap\mathcal{C}_{b}(\mathbb{R}_{+};L^2).$$
\end{theorem}

\begin{remark}
Compared with Theorem \ref{thm1}, the condition $\partial_z\omega_{0}\in L^2$ has been replaced by  $\omega_{0}\in L^{\infty}$ in Theorem \ref{lose-global}.
It enables us  to extend the global well-posedness theory to vector-field lying in space $L$ (which ensures the vector-field belongs to LogLip space) instead of being Lipschitz. Our choice is  motivated by
the well-known result that the velocity in the LogLip space $LL$ (see \eqref{LL} in Appendix \ref{appendix}) seems to be the minimal requirement for uniqueness to the incompressible Euler equations.
Indeed, the vorticity equation can provides us the $L^{p}$-norms of
vorticity with $2\leq p<\infty$. This allows us to get that $\nabla
u\in L^{\infty}_{\rm loc}(\RR_+;L)$ by means of the relation
$\norm{\nabla u}_{L^p}\leq Cp\norm{\omega}_{L^p} $. Furthermore, we
can obtain the global well-posedness by exploring losing estimates.
\end{remark}
 The paper is organized as follows.
   In Section \ref{section-pre} we shall give the definitions of the functional spaces
    that we shall use and  state some useful propositions and algebraic identity. Next,  we shall obtain a priori estimate for
       sufficiently smooth solutions of \eqref{eq1.1} in Section \ref{section-priori}.  The last two sections will be devoted to proving
       Theorems \ref{thm1} and \ref{lose-global}. In Appendix, we shall give a few technical lemmas  used throughout the paper. We shall also prove an existence result and a losing estimate for the anisotropic equations with a convection term, which are the key ingredients in the proof of the results.

\textbf{Notations:} Throughout the paper, we write $\mathbb{R}^{3}=\mathbb{R}^{2}_{h}\times\mathbb{R}_{v}$ The tridimensional vector field $u$ is denoted by $(u^h,u^z)$, and we agree that
$\nabla_h=(\partial_1,\partial_2)$. Finally, the $X_{h}$ (resp.,$X_v$) stands for that $X_{h}$ is a function space over $\mathbb{R}_{h}^2$ (resp.,$\RR_v$).
\section{Preliminaries }\label{section-pre}

\subsection {Littlewood-Paley Theory and Besov spaces}\label{space}
In this subsection, we provide the definition of some function spaces based on the
so-called Littlewood-Paley decomposition.

Let $(\chi,\varphi)$ be a couple of smooth functions with  values in $[0,1]$
such that $\chi$ is supported in the ball $\big\{\xi\in\mathbb{R}^{n}\big||\xi|\leq\frac{4}{3}\big\}$,
$\varphi$ is supported in the shell $\big\{\xi\in\mathbb{R}^{n}\big|\frac{3}{4}\leq|\xi|\leq\frac{8}{3}\big\}$ and
\begin{align}\label{one}
    \chi(\xi)+\sum_{j\in \mathbb{N}}\varphi(2^{-j}\xi)=1\quad {\rm for \ each\ }\xi\in \mathbb{R}^{n}.
\end{align}  
For every $u\in \mathcal{S}'(\mathbb{R}^{n})$, we define the dyadic blocks as
\begin{equation*}
 \Delta_{-1}u=\chi(D)u\quad\text{and}\quad   {\Delta}_{j}u:=\varphi(2^{-j}D)u\quad {\rm for\ each\ }j\in\mathbb{N}.
\end{equation*}
We shall also use the following low-frequency cut-off:
\begin{equation*}
    {S}_{j}u:=\chi(2^{-j}D)u.
\end{equation*}
One can easily show that
the formal equality
\begin{equation}\label{eq2.1}
    u=\sum_{j\geq-1}{\Delta}_{j}u
\end{equation}
holds in $\mathcal{S}'(\mathbb{R}^{n})$, and this is called the \emph{inhomogeneous Littlewood-Paley decomposition}.
It has nice properties of quasi-orthogonality:
\begin{equation}\label{eq2.2}
    {\Delta}_{j}{\Delta}_{j'}u\equiv 0\quad \text{if}\quad |j-j'|\geq 2.
\end{equation}
\begin{equation}\label{eq2.22}
{\Delta}_{j}({S}_{j'-1}u{\Delta}_{j'}v)\equiv0\quad \text{if}\quad |j-j'|\geq5.
\end{equation}

Next, we first introduce the Bernstein lemma which will be useful throughout this paper.
\begin{lemma}\label{bernstein}
 There exists a constant $C$ such that for $q,k\in\NN,$ $1\leq a\leq b$ and for  $f\in L^a(\RR^n)$,
\begin{eqnarray*}
\sup_{|\alpha|=k}\|\partial ^{\alpha}S_{q}f\|_{L^b}&\leq& C^k\,2^{q(k+n(\frac{1}{a}-\frac{1}{b}))}\|S_{q}f\|_{L^a},\\
\ C^{-k}2^
{qk}\|{\Delta}_{q}f\|_{L^a}&\leq&\sup_{|\alpha|=k}\|\partial ^{\alpha}{\Delta}_{q}f\|_{L^a}\leq C^k2^{qk}\|{\Delta}_{q}f\|_{L^a}.
\end{eqnarray*}

\end{lemma}
\begin{definition}\label{def2.2}
For $s\in \mathbb{R}$, $(p,q)\in [1,+\infty]^{2}$ and $u\in \mathcal{S}'(\mathbb{R}^{n})$, we set
\begin{equation*}
    \norm{u}_{{B}^{s}_{p,q}(\mathbb{R}^{n})}:=
    \Big(\sum_{j\geq-1}2^{jsq}\norm{{\Delta}_{j}u}_{L^{p}(\mathbb{R}^{n})}^{q}\Big)^{\frac{1}{q}}
    \quad\text{if}\quad r<+\infty
\end{equation*}
and
\begin{equation*}
    \norm{u}_{{B}^{s}_{p,\infty}(\mathbb{R}^{n})}:=\sup_{j\geq-1}2^{js}\norm{{\Delta}_{j}u}_{L^{p}(\mathbb{R}^{n})}.
\end{equation*}
Then we define the \emph{inhomogeneous Besov spaces} as
\begin{equation*}
   {B}^{s}_{p,q}(\mathbb{R}^{n}):=\big\{u\in\mathcal{S}'(\mathbb{R}^{n})\big|\norm{u}_{{B}^{s}_{p,q}(\mathbb{R}^{n})}<+\infty\big\}.
\end{equation*}
We also denote ${B}^{s}_{2,2}$ by $H^{s}$.
\end{definition}

Since the dissipation only occurs in the horizontal direction, it is natural to introduce the following
 definition.
\begin{definition}\label{def-anisotropic}
For $s,t\in \mathbb{R}$, $(p,q)\in [1,+\infty]^{2}$ and $u\in \mathcal{S}'(\mathbb{R}^{3})$, we set
\begin{equation*}
    \norm{u}_{{B}^{s,t}_{p,q}(\mathbb{R}^{3})}:=
  \Big(\sum_{j,k\geq-1}2^{jsq}2^{ktq}\norm{{\Delta}^{h}_{j}{\Delta}^{v}_{k}u}_{L^{p}(\mathbb{R}^{3})}^{q}\Big)^{\frac{1}{q}}
    \quad\text{if}\quad r<+\infty
\end{equation*}
and
\begin{equation*}
    \norm{u}_{{B}^{s,t}_{p,\infty}(\mathbb{R}^{3})}:=\sup_{j,k\geq-1}2^{js}2^{kt}\norm{{\Delta}^{h}_{j}{\Delta}^{v}_{k}u}_{L^{p}(\mathbb{R}^{3})}.
\end{equation*}
Then we define the \emph{anisotropic Besov spaces} as
\begin{equation*}
   {B}^{s,t}_{p,q}(\mathbb{R}^{3}):=\big\{u\in\mathcal{S}'(\mathbb{R}^{3})\big|\norm{u}_{{B}^{s,t}_{p,q}(\mathbb{R}^{3})}<+\infty\big\}.
\end{equation*}
We also denote ${B}^{s,t}_{2,2}$ by $H^{s,t}$.
\end{definition}

Let us now state some basic properties for $H^{s,t}$ spaces which will be useful later.
\begin{lemma}\label{properties}
The following properties of anisotropic Besov spaces hold:
\begin{enumerate}[$(i)$]
\item
Inclusion realtion: $\norm{u}_{H^{s_{2},t_{2}}(\mathbb{R}^{3})}\subseteq\norm{u}_{H^{s_{1},t_{1}}(\mathbb{R}^{3})}$ if $s_{2}\geq s_{1}$ and $t_{2}\geq t_{1}.$
    \item Interpolation: for $s_{1},s_{2},t_{1},t_{2}\in \mathbb{R}$ and $\theta\in[0,1]$, we have
    $$\norm{u}_{H^{\theta s_{1}+(1-\theta)s_{2},\theta t_{1}+(1-\theta)t_{2}}(\mathbb{R}^{3})}\leq
    \norm{u}^{\theta}_{H^{s_{2},t_{2}}(\mathbb{R}^{3})}\norm{u}^{1-\theta}_{H^{s_{1},t_{1}}(\mathbb{R}^{3})}.
    $$
  \item  For $s,t\geq 0$, $\norm{u}_{H^{s,t}(\mathbb{R}^{3})}$ is equivalent to $$\norm{u}_{L^{2}(\mathbb{R}^{3})}+\norm{\Lambda^{s}_{h}u}_{L^{2}(\mathbb{R}^{3})}
  +\norm{\Lambda^{t}_{v}u}_{L^{2}(\mathbb{R}^{3})}+\norm{\Lambda^{t}_{v}\Lambda^{s}_{h}u}_{L^{2}(\mathbb{R}^{3})}.$$
  \item $\norm{u}_{H^{s,t}(\mathbb{R}^{3})}\simeq\big\|\|u\|_{{H}^{s}(\mathbb{R}_{h}^{2})}\big\|_{{H}^{t}(\mathbb{R}_{v})}\simeq
  \big\|\|u\|_{{H}^{t}(\mathbb{R}_{v})}\big\|_{{H}^{s}(\mathbb{R}_{h}^{2})}.$
  \item Algebraic properties: for $s>1$ and $t>\frac 12$, $\norm{u}_{H^{s,t}(\mathbb{R}^{3})}$ is an algebra.
\end{enumerate}
\end{lemma}
\begin{proof}
We first point out that $(i)$ and $(ii)$ are obviously true. We only need to prove $(iii)$, $(iv)$ and $(v)$.
From the definition of \emph{ anisotropic Besov Spaces }$\norm{u}_{H^{s,t}(\mathbb{R}^{3})}$ and the Plancherel theorem, we can conclude that
\begin{equation*}
\begin{split}
 \norm{u}^{2}_{H^{s,t}(\mathbb{R}^{3})}=&
  \sum_{j,k\geq-1}2^{2js}2^{2kt}\norm{{\Delta}^{h}_{j}{\Delta}^{v}_{k}u}_{L^{2}(\mathbb{R}^{3})}^{2}
  = \sum_{j,k\geq-1}2^{2js}2^{2kt}\norm{{\varphi}^{2}_{hj}{\varphi}^{2}_{vk}\hat{u}}_{L^{2}(\mathbb{R}^{3})}^{2}\\
  =& \sum_{j,k\geq-1}2^{2js}\sum_{j,k\geq-1}2^{2kt}\int_{\mathbb{R}^{3}}{\varphi}^{2}_{hj}{\varphi}^{2}_{vk}|\hat{u}|^{2}(\xi)\mathrm{d}\xi\\
  =& \int_{\mathbb{R}_{h}^{2}}\sum_{j\geq-1}2^{2js}{\varphi}^{2}_{hj}\int_{\mathbb{R}_{v}}\sum_{k\geq-1}2^{2kt}
  {\varphi}^{2}_{vk}|\hat{u}|^{2}(\xi_{h},\xi_3)\mathrm{d}\xi_3\mathrm{d}\xi_{h}.
   \end{split}
\end{equation*}
This together with Equality \eqref{one} yields that
 \begin{equation}\label{eqmi}
\begin{split}
 \norm{u}^{2}_{H^{s,t}(\mathbb{R}^{3})} \simeq&\int_{\mathbb{R}^{3}}(1+\xi^2_h)^{s}(1+\xi^2_v)^{t}|\hat{u}|^{2}(\xi)\mathrm{d}\xi\\
 \simeq&\int_{\mathbb{R}^{3}}|\hat{u}|^{2}(\xi)\mathrm{d}\xi+\int_{\mathbb{R}^{3}}\xi^{2s}_h|\hat{u}|^{2}(\xi)\mathrm{d}\xi
  +\int_{\mathbb{R}^{3}}\xi_v^{2t}|\hat{u}|^{2}(\xi)\mathrm{d}\xi+\int_{\mathbb{R}^{3}}\xi_h^{2s}\xi_v^{2t}|\hat{u}|^{2}(\xi)\mathrm{d}\xi\\
   =&\norm{u}^{2}_{L^2(\mathbb{R}^{3})}+ \norm{\Lambda^{s}_h u}^{2}_{L^{2}(\mathbb{R}^{3})} +\norm{\Lambda^{t}_v u}^{2}_{L^{2}(\mathbb{R}^{3})} +\norm{\Lambda^{s}_{h}\Lambda^{t}_v u}^{2}_{L^{2}(\mathbb{R}^{3})}.
  \end{split}
\end{equation}
This implies the desired result $(iii)$.

From \eqref{eqmi},  it is clear that
\begin{equation*}
\begin{split}
\norm{u}^{2}_{H^{s,t}(\mathbb{R}^{3})}\simeq&\int_{\mathbb{R}^{3}}(1+\xi^2_h)^{s}(1+\xi^2_v)^{t}|\hat{u}|^{2}(\xi)\mathrm{d}\xi
\simeq\int_{\mathbb{R}^{3}}|(1+\Lambda^{s}_h)(1+\Lambda^{t}_v)u|^{2}(x)\mathrm{d}x\\
\simeq&\norm{(1+\Lambda^{2}_v)^\frac{t}{2}\|(1+\Lambda^{2}_h)^\frac{s}{2}u\|_{L_{h}^{2}}}^{2}_{L_{v}^{2}}
\simeq\big\|\|u\|_{H^{s}(\mathbb{R}_{h}^{2})}\big\|^{2}_{H^{t}(\mathbb{R}_{v})}.
\end{split}
\end{equation*}
Similarly, we can show that $\norm{u}_{H^{s,t}(\mathbb{R}^{3})}\simeq\big\|\|u\|_{{H}^{t}(\mathbb{R}_{v})}\big\|_{{H}^{s}(\mathbb{R}_{h}^{2})}.$

Finally, according to the fact that $H^{s}(\mathbb{R}^{n})(s>\frac n2)$ is an algebra and $(iv)$. Thus, for any $u,v\in H^{s,t}(s>1,t>\frac12)$,
\begin{align*}
\norm{uv}_{H^{s,t}}\leq& C\big\|\|uv\|_{{H}^{s}(\mathbb{R}_{h}^{2})}\big\|_{{H}^{t}(\mathbb{R}_{v})}\leq C\big\|\|u\|_{{H}^{s}(\mathbb{R}_{h}^{2})}\|v\|_{{H}^{s}(\mathbb{R}_{h}^{2})}\big\|_{{H}^{t}(\mathbb{R}_{v})}\\
\leq&C\norm{u}_{H^{s,t}}\norm{v}_{H^{s,t}}.
\end{align*}
This is exactly the last result.
\end{proof}

\subsection{Heat kernel and Algebraic Identity}
In this subsection, we first review the properties of the heat equation. Next, we give two useful algebraic identities and its properties.
\begin{proposition}\cite{che1,Lemar}
\label{heat}
There exist $c$ and $C>0$ such that  for
 every  $u$  solution of
\begin{equation*}
\begin{cases}
\partial_t u-\Delta u=0,\quad x\in \mathbb{R}^{n},\\
u|_{t=0}=u_{0},
\end{cases}
\end{equation*}
 the following estimates hold true
 \begin{enumerate}[$(i)$]
   \item  $
\|u (t) \|_{L^{p}(\mathbb{R}^{n})}=\|e^{t\Delta }u_{0}\|_{L^{p}(\mathbb{R}^{n})}\leq Ct^{-\frac{n}{2}(\frac{1}{q}-\frac{1}{p})}\|u_{0}\|_{L^{q}(\mathbb{R}^{n})}
$, for $1\leq q\leq p\leq\infty$.
   \item $
\|\Delta_{j}u (t) \|_{L^{p}(\mathbb{R}^{n})}=\|e^{t\Delta }\Delta_{j}u_{0}\|_{L^{p}(\mathbb{R}^{n})}\leq Ce^{-ct2^{2j}}\|u_{0}\|_{L^{p}(\mathbb{R}^{n})},
$ for $j\geq 0$.
 \end{enumerate}
\end{proposition}

Next, we intend to recall the behavior of the operator $\frac{\partial_r}{r}\Delta^{-1}$ over axisymmetric functions.
\begin{proposition}\label{prop1}\cite{hrou}
If $u$ is an axisymmetric smooth scalar function, then we have
\begin{equation}
\label{prop1-12}
\Big(\frac{\partial_r}{r}\Big)\Delta^{-1}u(x)=\frac{x_2^2}{r^2}\mathcal{R}_{11}u(x)+\frac{x_1^2}{r^2}\mathcal{R}_{22}u(x)-2\frac{x_1x_2}{r^2}\mathcal{R}_{12}u(x),
\end{equation}
with $\mathcal{R}_{ij}=\partial_{ij}\Delta^{-1}$. Moreover, for $p\in]1,\infty[$ there exists $C>0$ such that
\begin{equation}
\label{prop1-2}
\|({\partial_r}/{r})\Delta^{-1}u\|_{L^{p}}\le C\| u\|_{L^{p}}.
\end{equation}
\end{proposition}
\begin{proposition}\label{prop-identity}
Let $u$ be a free divergence axisymmetric vector-field without swirl and $\omega=\text{\rm curl}u$. Then
\begin{equation}\label{identity}
\frac{u^{r}}{r}=\partial_{z}\Delta^{-1}\left(\frac{\omega_\theta}{r}\right)-2\frac{\partial_r}{r}\Delta^{-1}\partial_z\Delta^{-1}\left(\frac{\omega_\theta}{r}\right).
\end{equation}
Besides, there hold that
\begin{equation}\label{ansitro}
\Big\|\partial_{z}\Big(\frac{u^{r}}{r}\Big)\Big\|_{L^{p}}\leq C\norm{\frac{\omega_{\theta}}{r}}_{L^{p}},\quad 1<p<\infty
\end{equation}
and
\begin{equation}\label{qianru}
\Big\|\frac{u^{r}}{r}\Big\|_{L^{\frac{3q}{3-q}}}\leq C\Big\|\frac{\omega_\theta}{r}\Big\|_{L^{q}},\quad 1<q<3.
\end{equation}
\end{proposition}
\begin{proof}
Using Biot-Savart law and the fact $\omega=\omega_\theta e_\theta$, we have
\begin{equation}\label{eq-2.5-1}
u^{1}=\Delta^{-1}\left((\partial_z \omega_\theta)\cos\theta\right)=\partial_z\Delta^{-1}\left(x_{1}\frac{\omega_{\theta}}{r}\right)
\end{equation}
and
\begin{equation}\label{eq-2.5-2}
u^{2}=\Delta^{-1}\left((\partial_z \omega_\theta)\sin\theta\right)=\partial_z\Delta^{-1}\left(x_{2}\frac{\omega_{\theta}}{r}\right).
\end{equation}
On the other hand, we observe that
\begin{equation}\label{eq2.14}
\Delta^{-1}\left(x_{i}\frac{\omega_{\theta}}{r}\right)=x_{i}\Delta^{-1}\left(\frac{\omega_{\theta}}{r}\right)-\big[x_{i},\Delta^{-1}\big]
\left(\frac{\omega_{\theta}}{r}\right),\quad \text{for }i=1,2.
\end{equation}
Applying the Laplace operator to the commutator $\big[x_{i},\Delta^{-1}\big]
\left(\frac{\omega_{\theta}}{r}\right)$, we get
\begin{equation*}
\begin{split}
\Delta\big[x_{i},\Delta^{-1}\big]\left(\frac{\omega_{\theta}}{r}\right)=&-x_i\frac{\omega_\theta}{r}+\Delta\left(x_{i}\Delta^{-1}\left(\frac{\omega_{\theta}}{r}\right)\right)\\
=&-x_i\frac{\omega_\theta}{r}+x_i\frac{\omega_\theta}{r}+2\partial_i x_i\partial_i\Delta^{-1}\left(\frac{\omega_{\theta}}{r}\right)\\
=&2\partial_i\Delta^{-1}\left(\frac{\omega_{\theta}}{r}\right).
\end{split}
\end{equation*}
This means that
\begin{equation*}
\big[x_{i},\Delta^{-1}\big]\left(\frac{\omega_{\theta}}{r}\right)=2\partial_i\Delta^{-2}\left(\frac{\omega_{\theta}}{r}\right)=2x_{i}\frac{\partial_r}{r}\Delta^{-2}\left(\frac{\omega_{\theta}}{r}\right).
\end{equation*}
Inserting this estimate in \eqref{eq2.14} gives
\begin{equation}\label{eq-2.5-3}
\Delta^{-1}\left(x_{i}\frac{\omega_{\theta}}{r}\right)=x_{i}\Delta^{-1}\left(\frac{\omega_{\theta}}{r}\right)-2x_{i}\frac{\partial_r}{r}\Delta^{-2}\left(\frac{\omega_{\theta}}{r}\right).
\end{equation}
Plugging \eqref{eq-2.5-3} in \eqref{eq-2.5-1} and \eqref{eq-2.5-2}, respectively, we get
\begin{equation*}
\begin{split}
\frac{u^{r}}{r}=\frac{x_{1}u^{1}+x_{2}u^{2}}{r^{2}}=&\frac{x^{2}_{1}+x^{2}_{2}}{r^{2}}\partial_z\Delta^{-1}\left(\frac{\omega_{\theta}}{r}\right)
-2\frac{x^{2}_{1}+x^{2}_{2}}{r^{2}}\frac{\partial_r}{r}\Delta^{-1}\partial_z\Delta^{-1}\left(\frac{\omega_{\theta}}{r}\right)\\
=&\partial_z\Delta^{-1}\left(\frac{\omega_{\theta}}{r}\right)
-2\frac{\partial_r}{r}\Delta^{-1}\partial_z\Delta^{-1}\left(\frac{\omega_{\theta}}{r}\right).
\end{split}
\end{equation*}
Furthermore,
\begin{equation*}
\partial_z\Big(\frac{u^{r}}{r}\Big)
=\partial^{2}_z\Delta^{-1}\left(\frac{\omega_{\theta}}{r}\right)
-2\frac{\partial_r}{r}\Delta^{-1}\partial^{2}_z\Delta^{-1}\left(\frac{\omega_{\theta}}{r}\right).
\end{equation*}
Combining this with Proposition \ref{prop1} ensures us to get the estimate \eqref{ansitro}.

Finally, by using Proposition \ref{prop1}, $L^p$-boundedness of Riesz operator and the Sobolev embedding theorem,
\begin{align*}
\Big\|\frac{u^{r}}{r}\Big\|_{L^{\frac{3q}{3-q}}}\leq C\norm{\partial_z\Delta^{-1}\left(\frac{\omega_{\theta}}{r}\right)}_{L^{\frac{3q}{3-q}}}
\leq C\Big\|\frac{\omega_\theta}{r}\Big\|_{L^{q}}.
\end{align*}
This completes the proof.
\end{proof}
Next,
we will give a precise expression about $u^r\over r$ and $\omega_\theta\over r$ by virtue of the algebraic identity \eqref{identity} and the harmonic analysis tools. More precisely:

\begin{proposition}\label{prop-sy}
Let $u$ be a free divergence axisymmetric vector-field without swirl and $\omega=\text{\rm curl}u$. Then
\begin{equation}\label{eq.2.14}
\frac{u^{r}}{r}=(c_{1}-4\gamma_{1}i)\frac{x_{3}}{|x|^{3}}\ast\frac{\omega_\theta}{r}+6\gamma_{1}i\frac{x^{2}_{2}
}{r^{2}}\frac{x^{2}_{1}x_{3}}{|x|^{5}}\ast\frac{\omega_\theta}{r}+6\gamma_{1}i\frac{x^{2}_{1}
}{r^{2}}\frac{x^{2}_{2}x_{3}}{|x|^{5}}\ast\frac{\omega_\theta}{r}-12\gamma_{1}i\frac{x_{1}
x_{2}}{r^{2}}\frac{x_{1}x_{2}x_{3}}{|x|^{5}}\ast\frac{\omega_\theta}{r},
\end{equation}
where $c_1=2\pi^{\frac{3}{2}}\frac{\Gamma(\frac 12)}{\Gamma(1)}$ and $\gamma_1=i\pi^{\frac 32}\frac{\Gamma(\frac{1}{2})}{\Gamma(2)}$.
\end{proposition}
\begin{proof}
From the algebraic identities \eqref{identity} and  \eqref{prop1-12}, we obtain
\begin{equation}\label{pidenity}
\frac{u^{r}}{r}=\partial_{z}\Delta^{-1}\left(\frac{\omega_\theta}{r}\right)-2\frac{x^{2}_{2}
}{r^{2}}\mathcal{R}_{11}\partial_z\Delta^{-1}\left(\frac{\omega_\theta}{r}\right)-2\frac{x^{2}_{1}
}{r^{2}}\mathcal{R}_{22}\partial_z\Delta^{-1}\left(\frac{\omega_\theta}{r}\right)+4\frac{x_{1}
x_{2}}{r^{2}}\mathcal{R}_{12}\partial_z\Delta^{-1}\left(\frac{\omega_\theta}{r}\right).
\end{equation}
On the one hand,
\begin{equation*}
\Delta^{-1}\left(\frac{\omega_\theta}{r}\right)=-\mathcal{F}^{-1}\Big(\frac{1}{|\xi|^{2}}
\Big(\frac{\omega_\theta}{r}\widehat{\Big)}(\xi)\Big)=-c_{1}\frac{1}{|x|}\ast\frac{\omega_\theta}{r}.
\end{equation*}
Thus,
\begin{equation*}
\partial_{z}\Delta^{-1}\left(\frac{\omega_\theta}{r}\right)=c_{1}\frac{x_{3}}{|x|^{3}}\ast\frac{\omega_\theta}{r}.
\end{equation*}
On the other hand,
\begin{equation}\label{eq.2.15}
\mathcal{R}_{kj}\partial_z\Delta^{-1}\left(\frac{\omega_\theta}{r}\right)=-i\partial_z\partial_{j}\mathcal{F}^{-1}\Big(\frac{\xi_{k}}{|\xi|^{4}}{
\Big(\frac{\omega_\theta}{r}}\widehat{\Big)}(\xi)\Big).
\end{equation}
Since $\xi_{k}$ is a harmonic polynomial of order one, then we can get by using Theorem 5 of {Chap-4} in \cite{S} that
\begin{equation*}
\mathcal{F}^{-1}\Big(\frac{\xi_{k}}{|\xi|^{4}}\Big)=\gamma_{1}\frac{x_{k}}{|x|}.
\end{equation*}
Inserting this equality to \eqref{eq.2.15}, we have
\begin{equation*}
\begin{split}
\mathcal{R}_{kj}\partial_z\Delta^{-1}\left(\frac{\omega_\theta}{r}\right)=&-i\gamma_{1}\partial_z\partial_{j}
\Big(\frac{x_{k}}{|x|}\ast\frac{\omega_\theta}{r}\Big)(x)\\
=&i\gamma_{1}\delta_{kj}\frac{x_{3}}{|x|^{3}}\ast\frac{\omega_\theta}{r}(x)-3i\gamma_{1}\frac{x_{k}x_{j}x_{3}}{|x|^{5}}\ast\frac{\omega_\theta}{r}(x).
\end{split}
\end{equation*}
Plugging these equalities in \eqref{pidenity} gives the desired result.
\end{proof}
\begin{remark}
Let us point out that the equality \eqref{eq.2.14} implies the S-Y estimate of \cite{Taira}. More precisely, by virtue of \eqref{eq.2.14} and the triangle inequality, we can conclude that
\begin{equation*}
\begin{split}
\Big|\frac{u^{r}}{r}\Big|=&\Big|(c_{1}-4\gamma_{1}i)\frac{x_{3}}{|x|^{3}}\ast\frac{\omega_\theta}{r}+6\gamma_{1}i\frac{x^{2}_{2}
}{r^{2}}\frac{x^{2}_{1}x_{3}}{|x|^{5}}\ast\frac{\omega_\theta}{r}+6\gamma_{1}i\frac{x^{2}_{1}
}{r^{2}}\frac{x^{2}_{2}x_{3}}{|x|^{5}}\ast\frac{\omega_\theta}{r}-12\gamma_{1}i\frac{x_{1}
x_{2}}{r^{2}}\frac{x_{1}x_{2}x_{3}}{|x|^{5}}\ast\frac{\omega_\theta}{r}\Big|\\
\leq&C\frac{1}{|x|^{2}}\ast\Big|\frac{\omega_\theta}{r}\Big|.
\end{split}
\end{equation*}
\end{remark}
\begin{proposition}\label{partial}
Let $u$ be a smooth  axisymmetric vector field with zero divergence  and we denote
$\omega=\omega_\theta e_\theta$. Then
$$
\Big\|\frac{u^r}{ r}\Big\|_{L^\infty}\le C\Big\|\frac{\omega_\theta}{r}\Big\|_{L^{2}}^{\frac12}\Big\|\nabla_h\Big(\frac{\omega_\theta}{r}\Big)\Big\|_{L^{2}}^{\frac12}.
$$
\end{proposition}
\begin{proof}

From \eqref{pidenity}, it is clear that
\begin{equation}\label{234}
\Big\|\frac{u^{r}}{r}\Big\|_{L^\infty(\mathbb{R}^{3})}\leq C\Big\|\partial_z\Delta^{-1}\Big(\frac{\omega_{\theta}}{r}\Big)\Big\|_{L^\infty(\mathbb{R}^{3})}
+C\sum^{2}_{k,j}\Big\|\mathcal{R}_{kj}\partial_z\Delta^{-1}\Big(\frac{\omega_{\theta}}{r}\Big)\Big\|_{L^\infty(\mathbb{R}^{3})}.
\end{equation}
For the first term in the right side of \eqref{234}, by using Lemma \ref{sharp}, we obtain
\begin{equation*}
\begin{split}
\Big\|\partial_z\Delta^{-1}\Big(\frac{\omega_{\theta}}{r}\Big)\Big\|_{L^\infty(\mathbb{R}^{3})}\leq & \Big\|\nabla\partial_z\Delta^{-1}\Big(\frac{\omega_{\theta}}{r}\Big)\Big\|^{\frac12}_{L^2(\mathbb{R}^{3})}
\Big\|\nabla_h\nabla\partial_z\Delta^{-1}\Big(\frac{\omega_{\theta}}{r}\Big)\Big\|^{\frac12}_{L^2(\mathbb{R}^{3})}\\
\leq &C\Big\|\frac{\omega_\theta}{r}\Big\|_{L^{2}(\mathbb{R}^{3})}^{\frac12}\Big\|\nabla_h\Big(\frac{\omega_\theta}{r}\Big)\Big\|_{L^{2}(\mathbb{R}^{3})}^{\frac12}.
\end{split}
\end{equation*}
Using Lemma \ref{sharp} again and applying the $L^p$-boundedness of Riesz operator, the second term can be bounded by
\begin{equation*}
\Big\|\nabla\mathcal{R}_{kj}\partial_z\Delta^{-1}\Big(\frac{\omega_{\theta}}{r}\Big)\Big\|^{\frac12}_{L^2(\mathbb{R}^{3})}
\Big\|\nabla_h\nabla\mathcal{R}_{kj}\partial_z\Delta^{-1}\Big(\frac{\omega_{\theta}}{r}\Big)\Big\|^{\frac12}_{L^2(\mathbb{R}^{3})}\leq C\Big\|\frac{\omega_\theta}{r}\Big\|_{L^{2}(\mathbb{R}^{3})}^{\frac12}\Big\|\nabla_h\Big(\frac{\omega_\theta}{r}\Big)\Big\|_{L^{2}(\mathbb{R}^{3})}^{\frac12}.
\end{equation*} This ends the proof.
\end{proof}

\section{The priori estimate }\label{section-priori}
In this section, we will give some useful priori estimates.
\subsection{Energy estimate and higher-order estimate}

We start with $L^{2}$ energy estimates and the maximum principle.
\begin{proposition}\label{Prop-Energy}
Let $(u,\rho)$ be a solution of \eqref{eq1.1}, then
\begin{equation}\label{energy}
\norm{u(t)}^{2}_{L^2}+\norm{\nabla_{h}u}_{L_{t}^2L^{2}}^{2}\leq (\norm{u_{0}}_{L^2}+t\norm{\rho_{0}}_{L^2})^{2}
\end{equation}
and
\begin{equation}\label{eq3.2}
\norm{\rho(t)}^{2}_{L^2}+\norm{\nabla_{h}\rho}_{L_{t}^2L^{2}}^{2}\leq \norm{\rho_{0}}^{2}_{L^2}.
\end{equation}
Besides, for $2\leq p\leq\infty$,
\begin{equation}\label{eq3.3}
\norm{\rho(t)}_{L^p}\leq\norm{\rho_{0}}_{L^p}.
\end{equation}
\end{proposition}
\begin{proof}
The proof is standard, we also give the proof for reader convenience.
We first prove the estimate \eqref{eq3.3}. Multiplying the second equation of \eqref{eq1.1} by $|\rho|^{p-2}\rho$ and integrating by parts yields that
 \begin{equation*}
 \frac{1}{p}\frac{\rm d}{{\rm d}t}\norm{\rho(t)}^{p}_{L^{p}}+(p-1)\int |\rho|^{p-2}|\nabla_{h}\rho|^{2}{\rm d}x=0.
 \end{equation*}
Thus we obtain
 \begin{equation*}
 \frac{\rm d}{{\rm d}t}\norm{\rho(t)}^{p}_{L^{p}}\leq 0,
 \end{equation*}
which implies immediately
$$\norm{\rho(t)}_{L^{p}}\leq\norm{\rho_{0}}_{L^{p}}.
$$
For $p=\infty$, it is just the maximum principle.

For the first one we take the $L^2$-inner  product of  the velocity equation with $u$.
 From  integration by parts and the  fact that $u$ is divergence free, we obtain
\begin{equation}\label{eqs1}
\frac12\frac{\rm d}{{\rm d}t}\|u(t)\|_{L^2}^2+\|\nabla_{h} u(t)\|_{L^2}^2\le\|u(t)\|_{L^2}\|\rho(t)\|_{L^2}.
\end{equation}
Furthermore, we conclude that
$$
\frac{\rm d}{{\rm d}t}\|u(t)\|_{L^2}\le\|\rho(t)\|_{L^2}.
$$
 By  integration  in time, we get that
$$
\|u(t)\|_{L^2}\le\|u_0\|_{L^2}+\int_0^t\|\rho(\tau)\|_{L^2}{ \rm d}\tau\le\|u_0\|_{L^2}+t\|\rho_{0}\|_{L^2},
$$
where we used the fact $\norm{\rho(t)}_{L^2}\leq\norm{\rho_{0}}_{L^2}$.
Plugging this estimate into \eqref{eqs1} yields
$$
\frac12\|u(t)\|_{L^2}^2+\int_0^t\|\nabla_{h} u(\tau)\|_{L^2}^2{\rm d}\tau
\le
\frac12\|u_0\|_{L^2}^2+\big(\|u_0\|_{L^2}+t\|\rho_0\|_{L^2}\big)\|\rho_0\|_{L^2} t.
$$
This implies the first result.

Finally, by the same argument as in proof of \eqref{energy}, we obtain the estimate \eqref{eq3.2}.
\end{proof}
Subsequently, we will establish the estimate of the quantities $\frac{\omega_\theta}{r}$ and $\omega$ which enable us to get the global existence of axisymmetric system \eqref{eq1.1}.
\begin{proposition}\label{Strong}
Assume that $u_0\in H^1,$ with $\frac{\omega_0}{r}\in L^2$ and $\rho_0\in L^{2}$. Let $(u,\rho)$ be a  smooth axisymmetric solution $(u,\rho)$ of  \eqref{eq1.1} without swirl, then we have

$$
\Big\| {\omega \over r} (t)\Big\|^{2}_{L^2}+\int_{0}^{t}\Big\| \nabla_{h}\left({\omega \over r}\right) (\tau)\Big\|^{2}_{L^2}{\rm d}\tau\le 2 \left(\norm{\frac{\omega_{0}}{r}}_{L^2}+\norm{\rho_{0}}_{2}\right)^{2},
$$
and
$$
\|u(t)\|_{H^1}^2+\int_0^t\|\nabla_{h}u(\tau)\|_{H^1}^2{\rm d}\tau\le C_{0} e^{C_0 t},
$$
where $C_0$ depends only  on the norm of the initial data.

\end{proposition}
\begin{proof}
 According to the equation \eqref{ww}, it is clear that $\frac{\omega_{\theta}}{r}$ satisfies the following equation
\begin{equation}
\label{equation_i}
\big(\partial_t+u\cdot\nabla\big)\frac{\omega_\theta}{r}-\big(\Delta_{h}+{{2 \over r}}\partial_r\big) \frac{\omega_\theta}{r} =-\frac{\partial_r\rho}{r}\cdot
\end{equation}
Now we recall that
$(\partial_{t}+u\cdot\nabla)\rho-\Delta_{h}\rho=0$, which can be rewritten as
\begin{equation}\label{eq-rtem}
(\partial_{t}+u\cdot\nabla)\rho-(\Delta_{h}+\frac{2}{r}\partial_{r})\rho=-\frac{2}{r}\partial_{r}\rho.
\end{equation}
In view of  \eqref{equation_i} and \eqref{eq-rtem}, we can set
 $\Gamma:= \frac{\omega_\theta}{r}-\frac{1}{2}\rho$
and then $\Gamma$ solves the equation
 \begin{equation*}
 (\partial_{t}+u\cdot\nabla)\Gamma-(\Delta_{h}+\frac{2}{r}\partial_{r})\Gamma=0.
 \end{equation*}
 Taking the $L^2$-inner product with $\Gamma$ and integrating by parts, we have
  \begin{equation*}\label{idI}
\frac12\frac{\rm d}{{\rm d}t}\|\Gamma(t)\|^{2}_{L^2}+\|\nabla_{h}\Gamma(t)\|_{L^2}^2\leq0,
\end{equation*}
where we used the facts that $u$ is divergence free and $-\int\frac{\partial_{r}\Gamma}{r}\Gamma {\mathrm d}x\geq 0$.
By  integration in time, we obtain that
$$
\|\Gamma(t)\|^{2}_{L^2}+\int_{0}^{t}\|\nabla_{h}\Gamma(\tau)\|_{L^2}^2\mathrm{d}\tau\leq\|\Gamma_{0}\|^{2}_{L^2}.
$$
This together with the estimate \eqref{eq3.2} yield that
\begin{equation*}
\begin{split}
&\norm{\frac{\omega_\theta}{r}(t)}^{2}_{L^{2}}+\int_{0}^{t}\Big\|\nabla_{h}\Big(\frac{\omega_\theta}{r}\Big)(\tau)\Big\|_{L^2}^2\mathrm{d}\tau\\
\leq&\big(\norm{\Gamma(t)}_{L^{2}}+\norm{\rho(t)}_{L^{2}}\big)^{2}+\big(\|\nabla_{h}\Gamma(t)\|_{L^{2}_{t}L^2}
+\|\nabla_{h}\rho(t)\|_{L^{2}_{t}L^2}\big)^2\\
\leq&2\big(\norm{\Gamma_{0}}_{L^{2}}+\norm{\rho_{0}}_{L^{2}}\big)^{2}.
\end{split}
\end{equation*}
This gives the first claimed estimate.

To prove the second estimate.
 By taking  the $L^2$-inner product of \eqref{tourbillon} with $\omega_\theta$
  we get
   \begin{equation*}
\frac12\frac{\rm d}{{\rm d }t}\|\omega_\theta(t)\|_{L^2}^2+\|\nabla_{h}\omega_\theta(t)\|_{L^2}^2+\left\|\frac{\omega_\theta}{r}(t)\right\|_{L^2}^2
=\int_{\RR^3}\frac{u^r}{r}\omega_\theta \omega_\theta \mathrm{d}x-\int_{\RR^3}\partial_r\rho\omega_\theta \mathrm{d}x.
\end{equation*}
Integrating by parts,
\begin{equation*}
\begin{split}
\int_{\RR^3}\partial_r\rho\omega_\theta \mathrm{d}x&=2\pi\int\partial_r\rho\omega_\theta r \mathrm{d}r\mathrm{d}z=2\pi\int\rho\partial_r\omega_\theta r \mathrm{d}r\mathrm{d}z+2\pi\int\rho\omega_\theta  \mathrm{d}r\mathrm{d}z\\
&=\int_{\mathbb{R}^{3}}\rho\partial_r\omega_\theta \mathrm{d}x+\int_{\mathbb{R}^{3}}\rho\frac{\omega_\theta}{r} \mathrm{d}x.
\end{split}
\end{equation*}
Thus, by the H\"older inequality, we have that
\begin{equation*}
\begin{split}
 \Big|\int_{\RR^3}\partial_r\rho\omega_\theta \mathrm{d}x \Big|  \leq&  \|\rho \|_{L^2}\, \big( \|\nabla_{h} \omega_{\theta} \|_{L^2}+
 \| \omega_{\theta}/r \|_{L^2}\big)\\
\leq & 2\norm{\rho_0}^{2}_{L^{2}}+\frac{1}{4}\big( \|\nabla_{h} \omega_{\theta} \|^{2}_{L^2}+
 \| \omega_{\theta}/r \|^{2}_{L^2}\big).
 \end{split}
\end{equation*}
Next, by virtue of the equality \eqref{a.11}, \mbox{Proposition \ref{prop-identity}} and the Young inequality, we obtain that
\begin{align*}
\left|\int_{\RR^3}\frac{u^{r}}{r}\omega_\theta\omega_\theta \mathrm{d}x\right|\leq & \Big\|\frac{u^{r}}{r}\Big\|^{\frac{3}{4}}_{L^{6}}\Big\|\partial_{z}\Big(\frac{u^{r}}{r}\Big)\Big\|^{\frac{1}{4}}_{L^{2}}
\norm{\omega_\theta}^{\frac{1}{2}}_{2}\norm{\nabla_{h}\omega_\theta}^{\frac{1}{2}}_{2}\norm{\omega_\theta}_{2}\\
\leq &C\Big\|\frac{\omega_{\theta}}{r}\Big\|_{L^{2}}\norm{\omega_\theta}^{\frac{3}{2}}_{2}\norm{\nabla_{h}\omega_\theta}^{\frac{1}{2}}_{2}\\
\leq &C\Big\|\frac{\omega_{\theta}}{r}\Big\|^{\frac{4}{3}}_{L^{2}}\norm{\omega_\theta}^{2}_{2}+\frac{1}{4}\norm{\nabla_{h}\omega_\theta}^{2}_{2}.
\end{align*}
Collecting these estimates with \mbox{Proposition \ref{Prop-Energy}} yield
\begin{equation*}
\label{fin1}
\begin{split}
\frac{\rm d}{{\rm d}t}\|\omega_\theta(t)\|_{L^2}^2+  \|\nabla_{h}\omega_\theta\|_{L^2}^2+\left\|\frac{\omega_\theta}{r}\right\|_{L^2}^2  \lesssim\norm{\rho_{0}}^{2}_{2}+\Big\|\frac{\omega_{\theta}}{r}\Big\|^{\frac{4}{3}}_{L^{2}}\norm{\omega_\theta}^{2}_{2}.
\end{split}
\end{equation*}
Therefore we get by the Gronwall inequality that
\begin{align*}
&\|\omega_\theta(t)\|_{L^2}^2+\int_0^t\Big(\|\nabla_{h}\omega_\theta(\tau)\|_{L^2}^2+\left\|\frac{\omega_\theta}{r}(\tau)\right\|_{L^2}^2\Big)\mathrm{d}\tau \\
\leq& Ce^{\int^{t}_{0}\|\frac{\omega_\theta}{r}(\tau)\|^{\frac{4}{3}}_{L^{2}}{\rm d}\tau}\Big(\Big\|\frac{\omega_\theta}{r}(0)\Big\|_{L^{2}}+\norm{\rho_0}_{L^{2}}t\Big).
\end{align*}
Since  $\| \omega \|_{L^2}= \| \omega_{\theta} \|_{L^2}$ and
$\|\nabla_{h}\omega\|_{L^2}^2=\|\nabla_{h}\omega_\theta\|_{L^2}^2+\left\|\frac{\omega_\theta}{r}\right\|_{L^2}^2.
$
So, we finally obtain that
$$
\|\omega(t)\|_{L^2}^2+\int_0^t\|\nabla_{h}\omega(\tau)\|_{L^2}^2\mathrm{d}\tau \le C_0e^{C_0 t}.
$$
This together with the energy estimates yields the second desired estimate. This ends the proof.
\end{proof}

\subsection{One derivative estimate on vertical variable}
In the absence of dissipation on vertical variable,
 we need to establish the following estimate on vertical variable in order to compensate this deficiency.
\begin{proposition}\label{vertical}
Assume that $\partial_{z}\rho_0\in L^{2}$ and $\partial_{z}\omega_0\in L^{2}$, then we have
\begin{equation}\label{eqv1}
\norm{\partial_{z}\rho(t)}^{2}_{L^{2}}+\int_{0}^{t}\norm{\nabla_{h}\partial_{z}\rho(\tau)}_{L^{2}}^{2}\mathrm{d}\tau\leq C_{1}e^{
\exp{C_{1}t}} ,
\end{equation}
and
\begin{equation}\label{eqv2}
\norm{\partial_{z}\omega(t)}^{2}_{L^{2}}+\int_{0}^{t}\norm{\nabla_{h}\partial_{z}\omega(\tau)}_{L^{2}}^{2}\mathrm{d}\tau\leq C_{2}e^{\exp{C_{2}t}},
\end{equation}
where $C_1$ and $C_2$ depend only  on the norm of the initial data $\rho_0$ and $\omega_0$.

\end{proposition}
\begin{proof}
Applying the operator $\partial_{z}$ to the second equation of \eqref{eq1.1}, we obtain
\begin{equation}\label{eq-vertical}
(\partial_{t}+u\cdot\nabla)\partial_{z}\rho-\Delta_{h}\partial_{z}\rho=-\partial_{z}u^{r}\partial_{r}\rho-\partial_{z}u^{z}\partial_{z}\rho.
\end{equation}
Taking the $L^{2}$-inner product of the equation \eqref{eq-vertical} with $\partial_{z}\rho$ and integrating by parts, we get
\begin{equation*}
\begin{split}
&\frac{1}{2}\frac{\rm d}{{\rm d}t}\norm{\partial_z\rho(t)}^{2}_{L^{2}}+\norm{\nabla_{h}\partial_z\rho(t)}^{2}_{L^{2}}=-\int\partial_{z}u^{r}\partial_{r}\rho\partial_z\rho \mathrm{d}x
-\int\partial_{z}u^{z}\partial_{z}\rho\partial_z\rho \mathrm{d}x\\
=&-\int\partial_{z}u^{r}\partial_{r}\rho\partial_z\rho \mathrm{d}x+\int\frac{u^r}{r}\partial_z \rho\partial_z \rho\mathrm{d}x+\int\partial_{r}u^r\partial_z \rho\partial_z \rho\mathrm{d}x\\
:=& I+II+III,
\end{split}
\end{equation*}
where we used the fact $\text{div}u=\partial_ru^{r}+\frac{u^r}{r}+\partial_zu^{z}=0.$

For the first term $I$, by using  \eqref{a.5}, the H\"older and the Young inequalities, we have
\begin{align*}
I\leq&\norm{\partial_zu^r}^{\frac{1}{2}}_{L^{2}}\norm{\nabla_{h}\partial_zu^r}^{\frac{1}{2}}_{L^{2}}\norm{\partial_r\rho}_{L^{2}}^{\frac{1}{2}}
\norm{\partial_z\partial_r\rho}_{L^{2}}^{\frac{1}{2}}\norm{\partial_z\rho}^{\frac{1}{2}}_{L^{2}}\norm{\nabla_h\partial_z\rho}^{\frac{1}{2}}_{L^{2}}\\
\leq&2\norm{\partial_zu^r}_{L^{2}}\norm{\nabla_{h}\partial_zu^r}_{L^{2}}\norm{\partial_r\rho}_{L^{2}}
\norm{\partial_z\rho}_{L^{2}}+\frac{1}{4}\norm{\nabla_h\partial_z\rho}^{2}_{L^{2}}\\
\leq&\norm{\nabla_{h}\partial_zu^r}^{2}_{L^{2}}+\norm{\omega}^{2}_{L^{2}}\norm{\partial_r\rho}^{2}_{L^{2}}
\norm{\partial_z\rho}^{2}_{L^{2}}+\frac{1}{4}\norm{\nabla_h\partial_z\rho}^{2}_{L^{2}}.
\end{align*}
We now turn to bound the term $II$, by using  \eqref{a.11}, Proposition \ref{prop-identity} and the Young inequality, we have
\begin{equation*}
\begin{split}
II\leq & \norm{u^r/r}^{\frac{3}{4}}_{L^{6}}\norm{\partial_z(u^r/r)}^{\frac{1}{4}}_{L^{2}}
\norm{\partial_z\rho}^{\frac12}_{L^{2}}\norm{\nabla_h\partial_z\rho}^{\frac12}_{L^{2}}\norm{\partial_z\rho}_{L^{2}}\\
\leq &C\norm{\omega_\theta/r}_{L^{2}}\norm{\partial_z\rho}^{\frac32}_{L^{2}}\norm{\nabla_h\partial_z\rho}^{\frac12}_{L^{2}}
\leq C\norm{\omega_\theta/r}^{\frac{4}{3}}_{L^{2}}\norm{\partial_z\rho}^{2}_{L^{2}}+\frac{1}{4}\norm{\nabla_h\partial_z\rho}^{2}_{L^{2}}.
\end{split}
\end{equation*}
Similarly, the term $III$ can be bounded by
\begin{equation*}
\begin{split}
& \norm{\partial_r u^{r}}^{\frac{3}{4}}_{L^{6}}\norm{\partial_z\partial_r u^{r}}^{\frac{1}{4}}_{L^{2}}\norm{\partial_z\rho}^{\frac32}_{L^{2}}
 \norm{\nabla_h\partial_z\rho}^{\frac12}_{L^{2}}\\
 \leq &C\norm{\partial_r\nabla u}_{L^{2}}\norm{\partial_z\rho}^{\frac32}_{L^{2}}\norm{\nabla_h\partial_z\rho}^{\frac12}_{L^{2}}
\leq C\norm{\nabla_h\omega}^{\frac{4}{3}}_{L^{2}}\norm{\partial_z\rho}^{2}_{L^{2}}+\frac{1}{4}\norm{\nabla_h\partial_z\rho}^{2}_{L^{2}}.
\end{split}
\end{equation*}
Combining these estimates, we have
\begin{equation*}
\begin{split}
&\frac{\rm d}{{\rm d}t}\norm{\partial_z\rho(t)}^{2}_{L^{2}}+\norm{\nabla_{h}\partial_z\rho(t)}^{2}_{L^{2}}\\ \lesssim&\norm{\nabla_{h}\partial_zu^r}^{2}_{L^{2}}+\norm{\omega}^{2}_{L^{2}}\norm{\partial_r\rho}^{2}_{L^{2}}
\norm{\partial_z\rho}^{2}_{L^{2}}+\norm{\nabla_h\omega}^{\frac{4}{3}}_{L^{2}}\norm{\partial_z\rho}^{2}_{L^{2}}+\Big\|\frac{\omega_{\theta}}{r}\Big\|^{\frac{4}{3}}_{L^{2}}\norm{\partial_z\rho}^{2}_{L^2}.
\end{split}
\end{equation*}
Since $\norm{\nabla_h\nabla u}_{L^{2}}\simeq\norm{\nabla_h\omega}_{L^{2}}$. By the Gronwall inequality and Proposition \ref{Strong}, we obtain the first desired result \eqref{eqv1}.

Applying $\partial_{z}$ to the equation \eqref{tourbillon-0}, we get
\begin{equation*}
 \label{tourbillon-vertical}
\partial_t \partial_{z}\omega +u\cdot\nabla\partial_{z}\omega-\Delta_{h}\partial_{z}\omega
 =-\partial^{2}_{zr}\rho e_{\theta}+\frac{\partial_{z}u^r}{r}\omega +\frac{u^r}{r}\partial_{z}\omega-\partial_{z}u^{r}\partial_{r}\omega-\partial_{z}u^{z}\partial_{z}\omega.
\end{equation*}
Taking the $L^{2}$-inner product to the above equation with $\partial_{z}\omega$ and integrating by parts, we obtain
\begin{align*}
&\frac{1}{2}\frac{\rm d}{{\rm d}t}\norm{\partial_z\omega(t)}^{2}_{L^{2}}+\norm{\nabla_{h}\partial_z\omega(t)}^{2}_{L^{2}}\\
=&-\int\partial^{2}_{zr}\rho e_{\theta}\partial_{z}\omega \mathrm{d}x+\int\frac{\partial_{z}u^r}{r}\omega \partial_{z}\omega \mathrm{d}x+\int \frac{u^r}{r}\partial_{z}\omega\partial_z\omega \mathrm{d}x\\
&-\int\partial_{z}u^{r}\partial_{r}\omega\partial_z\omega \mathrm{d}x
-\int\partial_{z}u^{z}\partial_{z}\omega\partial_z\omega \mathrm{d}x\\
=&-\int\partial^{2}_{zr}\rho e_{\theta}\partial_{z}\omega \mathrm{d}x+\int\partial_{z}\left(\frac{u^{r}}{r}\right)\omega \partial_{z}\omega \mathrm{d}x+2\int \frac{u^r}{r}\partial_{z}\omega\partial_z\omega \mathrm{d}x
\\&
-\int\partial_{z}u^{r}\partial_{r}\omega\partial_z\omega \mathrm{d}x
+\int\partial_{r}u^{r}\partial_{z}\omega\partial_z\omega \mathrm{d}x\\
:=&\sum^{5}_{i=1}J_{i}.
\end{align*}
Here we used the fact ${\rm div}u=\partial_ru^r+\frac{u^r}{r}+\partial_zu^z=0$.

By the H\"older inequality and the Cauchy-Schwarz inequality, we know
$$
J_1 \leq\norm{\partial^{2}_{zr}\rho}_{L^{2}}\norm{\partial_z \omega}_{L^{2}}\leq \norm{\partial_z \omega}^{2}_{L^{2}}+\norm{\partial^{2}_{zr}\rho}^{2}_{L^{2}}.
$$
By the inequality \eqref{a.11}  and Proposition \ref{prop-identity},
\begin{align*}
J_{2}\leq&\Big\|\partial_z\Big(\frac{u^{r}}{r}\Big)\Big\|_{L^2}\norm{\omega}^{\frac{3}{4}}_{L^{6}}\norm{\partial_z\omega}^{\frac{1}{4}}_{L^{2}}
\norm{\partial_z\omega}^{\frac{1}{2}}_{L^{2}}\norm{\nabla_h\partial_z\omega}^{\frac{1}{2}}_{L^{2}}\\
\leq&C\Big\|\partial_z\Big(\frac{u^{r}}{r}\Big)\Big\|_{L^2}\norm{\nabla\omega}^{\frac{3}{4}}_{L^{2}}\norm{\partial_z\omega}^{\frac{3}{4}}_{L^{2}}
\norm{\nabla_h\partial_z\omega}^{\frac{1}{2}}_{L^{2}}\\
\leq&C\Big\|\frac{\omega_\theta}{r}\Big\|_{L^2}\norm{\nabla_{h}\omega}^{\frac{3}{4}}_{L^{2}}\norm{\partial_z\omega}^{\frac{3}{4}}_{L^{2}}
\norm{\nabla_h\partial_z\omega}^{\frac{1}{2}}_{L^{2}}+C\Big\|\frac{\omega_\theta}{r}\Big\|_{L^2}\norm{\partial_z\omega}^{\frac{3}{2}}_{L^{2}}
\norm{\nabla_h\partial_z\omega}^{\frac{1}{2}}_{L^{2}}\\
\leq &C\Big\|\frac{\omega_\theta}{r}\Big\|^{\frac{4}{3}}_{L^2}\norm{\nabla_{h}\omega}^{2}_{L^{2}}+
C\Big\|\frac{\omega_\theta}{r}\Big\|^{\frac{4}{3}}_{L^2}\norm{\partial_{z}\omega}^{2}_{L^{2}}+\frac{1}{8}\norm{\nabla_h\partial_z\omega}^{2}_{L^{2}}.
\end{align*}
For the third term $J_3$,
we get by virtue of the inequality \eqref{a.11} and Proposition \ref{prop-identity} that
\begin{align*}
J_3 \leq&2\Big\|\frac{u^{r}}{r}\Big\|^{\frac{3}{4}}_{L^{6}}\Big\|\partial_{z}\Big(\frac{u^{r}}{r}\Big)\Big\|^{\frac{1}{4}}_{L^{2}}
\norm{\partial_{z}\omega}^{\frac{3}{2}}_{L^{2}}
\norm{\nabla_{h}\partial_{z}\omega}_{L^{2}}\\
\leq&C\norm{\frac{\omega}{r}}^{\frac{4}{3}}_{L^{2}}
\norm{\partial_{z}\omega}^{2}_{L^{2}}
+\frac{1}{8}\norm{\nabla_{h}\partial_{z}\omega}^{2}_{L^{2}}.
\end{align*}
Arguing as for proving  $J_2$, the term  $J_4$ can be bounded as follows.
\begin{align*}
J_{4}\leq&\norm{\partial_r\omega}^{\frac{1}{2}}_{L^{2}}\norm{\partial_{z}\partial_r\omega}^{\frac{1}{2}}_{L^{2}}\norm{\partial_{z}u^{r}}^{\frac{1}{2}}_{L^{2}}
\norm{\nabla_{h} \partial_{z}u^{r}}^{\frac{1}{2}}_{L^{2}}\norm{\partial_{z}\omega}^{\frac{1}{2}}_{L^{2}}
\norm{\nabla_{h}\partial_{z}\omega}^{\frac{1}{2}}_{L^{2}}\\
\leq&\norm{\omega}^{\frac{1}{2}}_{L^2}\norm{\partial_r\omega}^{\frac{1}{2}}_{L^{2}}\norm{\nabla_{h} \partial_{z}u^{r}}^{\frac{1}{2}}_{L^{2}}\norm{\partial_{z}\omega}^{\frac{1}{2}}_{L^{2}}\norm{\nabla_{h}\partial_{z}\omega}_{L^{2}}\\
\leq&C\norm{\omega}_{L^2}\norm{\partial_r\omega}_{L^{2}}\norm{\nabla_{h} \partial_{z}u^{r}}_{L^{2}}\norm{\partial_{z}\omega}_{L^{2}}+\frac{1}{8}\norm{\nabla_{h}\partial_{z}\omega}^{2}_{L^{2}}\\
\leq&C\norm{\omega}^{2}_{L^2}\norm{\nabla_{h} \partial_{z}u^{r}}^{2}_{L^{2}}\norm{\partial_{z}\omega}^{2}_{L^{2}}+C\norm{\partial_r\omega}^{2}_{L^{2}}+\frac{1}{8}\norm{\nabla_{h}\partial_{z}\omega}^{2}_{L^{2}}.
\end{align*}
We turn to bound the term $J_5$, by Lemma \ref{lema.1} and the Young inequality, we obtain
\begin{equation*}
\begin{split}
J_{5}\leq&\norm{\partial_1 u^{1}}^{\frac{1}{2}}_{L^{2}}\norm{\partial_{z}\partial_1 u^{1}}^{\frac{1}{2}}_{L^{2}}\norm{\partial_{z}\omega}_{L^{2}}
\norm{\nabla_{h}\partial_{z}\omega}_{L^{2}}\\
\leq&\norm{\omega}_{L^{2}}\norm{\partial_{z}\partial_1 u^{1}}_{L^{2}}\norm{\partial_{z}\omega}^{2}_{L^{2}}
+\frac{1}{8}\norm{\nabla_{h}\partial_{z}\omega}_{L^{2}}.
\end{split}
\end{equation*}
Putting this all together and  using the fact that $\norm{\nabla_h\nabla u}_{L^{2}}\simeq\norm{\nabla_h\omega}_{L^{2}}$, we get
\begin{align*}
&\frac{\rm d}{{\rm d}t}\norm{\partial_z\omega(t)}^{2}_{L^{2}}+\frac{3}{4}\norm{\nabla_{h}\partial_z\omega(t)}^{2}_{L^{2}}\\
\leq&\norm{\partial_z \omega}^{2}_{L^{2}}+\norm{\partial^{2}_{zr}\rho}^{2}_{L^{2}}+C\Big\|\frac{\omega_\theta}{r}\Big\|^{\frac{4}{3}}_{L^2}\norm{\nabla_{h}\omega}^{2}_{L^{2}}+
C\Big\|\frac{\omega_\theta}{r}\Big\|^{2}_{L^2}\norm{\partial_{z}\omega}^{2}_{L^{2}}\\
&+C\norm{\omega}^{2}_{L^2}\norm{\nabla_{h} \partial_{z}u^{r}}^{2}_{L^{2}}\norm{\partial_{z}\omega}^{2}_{L^{2}}+C\norm{\partial_r\omega}^{2}_{L^{2}}+\norm{\omega}_{L^{2}}\norm{\nabla_{h}\nabla u}_{L^{2}}\norm{\partial_{z}\omega}^{2}_{L^{2}}\\
\lesssim &\Big(1+\Big\|\frac{\omega_\theta}{r}\Big\|^{\frac{4}{3}}_{L^2}+\norm{\omega}^{2}_{L^2}\norm{\nabla_{h} \omega}^{2}_{L^{2}}\Big)\norm{\partial_z \omega}^{2}_{L^{2}}+\norm{\partial^{2}_{zr}\rho}^{2}_{L^{2}}+\Big\|\frac{\omega_\theta}{r}\Big\|^{\frac{4}{3}}_{L^2}\norm{\nabla_{h}\omega}^{2}_{L^{2}}
+\norm{\partial_r\omega}^{2}_{L^{2}}.
\end{align*}
This together with  Proposition \ref{Strong} and the Gronwall inequality yields the desired the result.
\end{proof}
\subsection{Strong a priori estimate}
In the following, our target is to establish the global estimate about Lipschitz norm of the velocity which ensures the global existence of solution.

Let us first give a useful lemma which provides the maximal smooth effect of the velocity in horizontal direction.
\begin{lemma}\label{smoothing}
Let $s_{1},s_{2}\in \RR$ and $p\in[2,\infty[$. Assume that $(u, \rho)$ be a smooth  solution  of the system \eqref{eq1.1}, then there holds that
\begin{equation}
\norm{u}_{L^{1}_{t}B^{s_{1}+2,s_{2}}_{p,1}}
\lesssim  \norm{u_{0}}_{B_{p,1}^{s_{1},s_{2}}}+\norm{u}_{L^{1}_{t}B_{p,1}^{s_{1},s_{2}}}+\norm{u\otimes u}_{L^{1}_{t}B_{p,1}^{s_{1}+1,s_{2}}\cap L^{1}_{t}B_{p,1}^{s_{1},s_{2}+1}}+\norm{\rho}_{L^{1}_{t}B_{p,1}^{s_{1},s_{2}}}.
\end{equation}
\end{lemma}
\begin{proof}
Applying the operator $\Delta^{h}_{q}\Delta^{v}_{k}$ to \eqref{eq1.1} and using Duhamel formula we get
\begin{equation*}
\begin{split}
u_{q,k}(t)=&e^{t\Delta_{h}}u_{q,k}(0)-\int_0^te^{(t-\tau)\Delta_{h}}\Delta^{h}_q\Delta^{v}_k\mathcal{P}(u\cdot\nabla u)(\tau,x)\mathrm{d}\tau-\int_0^te^{(t-\tau)\Delta_{h}}\Delta^{h}_q\Delta^{v}_k\mathcal{P}\rho(\tau,x)e_z\mathrm{d}\tau,
\end{split}
\end{equation*}
where $u_{q,k}=\Delta^{h}_q\Delta^{v}_k u$ and $\mathcal{P}$ is the Leray projection on divergence free vector fields.

According to Proposition \ref{heat}, we have the following estimate for $q\geq0$
$$
\|e^{t\Delta_{h}}\Delta^{h}_q \Delta^{v}_k f\|_{L^p(\RR^3)}\leq\norm{\|e^{t\Delta_{h}}\Delta^{h}_q \Delta^{v}_k f\|_{L^p(\RR_{h}^2)}}_{L^{p}(\RR_{v})}\le Ce^{-ct2^{2q}}\|\Delta^{h}_q \Delta^{v}_k f\|_{L^p(\RR^3)}.
$$
Therefore, for $q\geq 0$, we have
\begin{equation*}
\|u^{r}_{q,k}\|_{L^1_tL^p}\lesssim 2^{-2q}\|u_{q,k}(0)\|_{L^p}+2^{-2q}\int_0^t\|\Delta^{h}_q \Delta^{v}_k(u\cdot\nabla u)(\tau)\|_{L^p}\mathrm{d}\tau+2^{-2q}\|\rho_{q,k}\|_{L^1_t L^p}
\end{equation*}
Multiplying $2^{q(s_1+2)}2^{ks_2}$ and summing over $q,k$, we obtain
\begin{align*}
&\sum^{+\infty}_{q=0,k=-1}2^{q(s_{1}+2)}2^{ks_{2}}\|u_{q,k}\|_{L^1_tL^p}\\
\lesssim & \sum^{+\infty}_{q=0,k=-1}2^{qs_{1}}2^{ks_{2}}\|u_{q,k}(0)\|_{L^p}+\int_0^t\sum^{+\infty}_{q=0,k=-1}2^{q(s_{1}+1)}2^{ks_{2}}\|\Delta^{h}_q \Delta^{v}_k(u\otimes u)(\tau)\|_{L^p}\mathrm{d}\tau\\&+\int_0^t\sum^{+\infty}_{q=0,k=-1}2^{qs_{1}}2^{k(s_{2}+1)}\|\Delta^{h}_q \Delta^{v}_k(u\otimes u)(\tau)\|_{L^p}\mathrm{d}\tau+\int_{0}^{t}\sum^{+\infty}_{q=0,k=-1}2^{qs_{1}}2^{ks_{2}}\|\rho_{q,k}(\tau)\|_{ L^p}\mathrm{d}\tau
\end{align*}
It follows that,
\begin{equation*}
\begin{split}
\norm{u}_{L^{1}_{t}B^{s_{1}+2,s_{2}}_{p,1}}\lesssim & \norm{u}_{L^{1}_{t}B_{p,1}^{s_{1},s_{2}}}+\sum^{+\infty}_{q=0,k=-1}2^{q(s_{1}+2)}2^{ks_{2}}\|u_{q,k}\|_{L^1_tL^p}\\
\lesssim & \norm{u}_{L^{1}_{t}B_{p,1}^{s_{1},s_{2}}}+\norm{u_{0}}_{B_{p,1}^{s_{1},s_{2}}}+\norm{u\otimes u}_{L^{1}_{t}B_{p,1}^{s_{1}+1,s_{2}}}+\norm{u\otimes u}_{L^{1}_{t}B_{p,1}^{s_{1},s_{2}+1}}+\norm{\rho}_{L^{1}_{t}B_{p,1}^{s_{1},s_{2}}}.
\end{split}
\end{equation*}
This ends the proof.
\end{proof}
\begin{proposition}\label{Lipschitz}
Let $u_0\in H^1$ be a divergence free  axisymmetric without swirl vector field such that  $\frac{\omega_0}{r}\in L^2,~\partial_z\omega_{0}\in L^{2}$ and  $\rho_0\in  H^{0,1}$ an axisymmetric  function. Then any smooth  solution  $(u, \rho)$ of the system \eqref{eq1.1}
satisfies
$$
\|\nabla u\|_{L^1_tL^\infty}\le C_0e^{\exp{C_0 t}}.
$$
Here, the constant $C_{0}$ depends on the initial data.
\end{proposition}
\begin{proof}
According to the structure of axisymmetric flows and the incompressible property of velocity, we know that
${\rm div }u=\partial_ru^r+\frac{u^r}{r}+\partial_zu^z=0$ and  $\omega_\theta=\partial_zu^r-\partial_ru^z$. Therefore
\begin{equation}
\begin{split}
\|\nabla u\|_{L^1_tL^\infty}\leq& \|\partial_r u^{r}\|_{L^1_tL^\infty}+\|\partial_z u^{r}\|_{L^1_tL^\infty}+\|\partial_r u^{z}\|_{L^1_tL^\infty}+\|\partial_z u^{z}\|_{L^1_tL^\infty}\\
\lesssim & \Big\|\frac{u^{r}}{r}\Big\|_{L^1_tL^\infty}+\|\partial_r u^{r}\|_{L^1_tL^\infty}+\|\partial_z u^{r}\|_{L^1_tL^\infty}+\|\partial_r u^{z}\|_{L^1_tL^\infty}.
\end{split}
\end{equation}
For the quantity $\big\|\frac{u^{r}}{r}\big\|_{L^1_tL^\infty}$. By virtue of \mbox{Proposition \ref{partial}} and \mbox{Proposition \ref{Strong}}, we get that
\begin{equation*}
\Big\|\frac{u^{r}}{r}\Big\|_{L^1_tL^\infty}\leq C\Big\|\frac{\omega_{\theta}}{r}\Big\|^{\frac12}_{L^\infty_tL^2}\Big\|\nabla_h\Big(\frac{\omega_{\theta}}{r}\Big)\Big\|^{\frac12}_{L^1_tL^2}
\leq Ce^{Ct}.
\end{equation*}
Next, we turn to bound the quantity $\|\partial_z  u^{r}\|_{L^1_tL^\infty}$, by using \mbox{Lemma \ref{sharp}} and the Bernstein inequality, we have
\begin{equation*}
\|\partial_z u^{r}\|_{L^1_tL^\infty}\leq C\int_{0}^{t}\|\partial_z \nabla u^{r}\|^{\frac12}_{L^2}\|\nabla_h\partial_z \nabla u^{r}\|^{\frac12}_{L^2}{\mathrm d}\tau
\leq C\|\partial_z \omega\|^{\frac12}_{L^\infty_tL^2}\|\nabla_h\partial_z \omega\|^{\frac12}_{L^1_tL^2}.
\end{equation*}
For the quantity $\|\partial_r u^{r}\|_{L^1_tL^\infty}$ and $\|\partial_r u^{z}\|_{L^1_tL^\infty},$  by taking advantage of Lemma \ref{bernstein}, we know
$$\|\partial_r u^{r}\|_{L^1_tL^\infty}+\|\partial_r u^{z}\|_{L^1_tL^\infty}\leq C\norm{u}_{L^1_tB_{2,1}^{2,\frac12}}.
$$
Furthermore, by virtue of Lemma \ref{smoothing}, we get
\begin{align*}
\norm{u}_{L^1_tB_{2,1}^{2,\frac12}}\lesssim&\norm{u_{0}}_{B^{0,\frac12}_{2,1}}+\norm{u}_{L^{1}_{t}B_{2,1}^{0,\frac12}}+\norm{u\otimes u}_{L^{1}_{t}B_{2,1}^{1,\frac12}}+\norm{u\otimes u}_{L^{1}_{t}B_{2,1}^{0,\frac32}}+\norm{\rho}_{L^{1}_{t}B_{2,1}^{0,\frac12}}\\
\lesssim&\norm{u_{0}}_{H^{1,1}}+\norm{u}_{L^{1}_{t}H^{1,1}}+\norm{u\otimes u}_{L^{1}_{t}H^{2,1}}+\norm{u\otimes u}_{L^{1}_{t}{H}^{\frac54,\frac74}}+\norm{\rho}_{L^{1}_{t}H^{1,1}}\\
\lesssim&\norm{u_{0}}_{H^{1,1}}+\norm{u}_{L^{1}_{t}H^{1,1}}+\norm{u}^{2}_{L^{2}_{t}H^{2,1}}+
\norm{u}^{2}_{L^{2}_{t}H^{\frac54,\frac74}}+\norm{\rho}_{L^{1}_{t}H^{1,1}}.
\end{align*}
On the other hand, from the definition of space, we have
\begin{equation*}
\norm{u}_{L^2_tH^{2,1}}\lesssim\norm{u}_{L^2_tL^{2}}+\norm{\partial_zu}_{L^2_tL^{2}}+\norm{\nabla^{2}_hu}_{L^2_tL^{2}}+\norm{\nabla^{2}_h\partial_zu}_{L^2_tL^{2}}
\end{equation*}
and
\begin{equation*}
\begin{split}
\norm{u}_{L^2_tH^{\frac54,\frac74}}\lesssim&\norm{u}_{L^2_tL^{2}}+\big\|\Lambda^{\frac54}_{h}u\big\|_{L^2_tL^{2}}+\big\|\Lambda^{\frac74}_{v}u\big\|_{L^2_tL^{2}}
+\big\|\Lambda^{\frac54}_{h}\Lambda^{\frac74}_{v}u\big\|_{L^2_tL^{2}}\\
\lesssim&\norm{u}_{L^2_tL^{2}}+\big\|\nabla_{h}u\big\|_{L^2_tL^{2}}+\big\|\nabla^{2}_{h}u\big\|_{L^2_tL^{2}}+\big\|\partial_{z}u\big\|_{L^2_tL^{2}}+\big\|\partial^{2}_{z}u\big\|_{L^2_tL^{2}}
+\big\|\nabla_{h}\partial_{z}\omega\big\|_{L^2_tL^{2}}.
\end{split}
\end{equation*}
It remains to bound the norm of $\rho$. By the first estimate of \mbox{Proposition \ref{Strong}} and \mbox{Proposition \ref{vertical}}, we have
\begin{equation*}
\norm{\rho}_{L^1_tH^{1,1}}\lesssim \big\|\rho\big\|_{L^1_tL^{2}}+
\big\|\partial_z\rho\big\|_{L^1_tL^{2}}+\big\|\nabla_h\rho\big\|_{L^1_tL^{2}}+\big\|\nabla_h\partial_z\rho\big\|_{L^1_tL^{2}}\leq  C e^{\exp{Ct}}.
\end{equation*}
Collecting these estimates with \mbox{Proposition \ref{energy}}, \mbox{Proposition \ref{Strong}} and \mbox{Proposition \ref{vertical}} yields that
\begin{equation*}\label{line}
\|\nabla u\|_{L^1_tL^\infty}\le C_0 e^{\exp{C_0 t}}.
\end{equation*}
 This ends the proof.
\end{proof}

\section{Proof of Theorem \ref{thm1}}\label{sectionproof}
Here we use the Friedrichs method (see \cite{dp2} for more
details): For $n\geq 1$, let $J_n$ be the spectral cut-off defined by
\begin{equation*}
\widehat{J_{n}f}(\xi)=1_{[0,n]}(|\xi|)\widehat{f}(\xi), \quad \xi \in\RR^3.
\end{equation*}
We consider the following system in the spaces $L^{2}_{n}:=\{f\in L^2(\RR^{3})|\text{ supp} f\subset B(0,n)\}$:
\begin{equation}
\begin{cases}
\partial_tu+\mathcal{P}J_{n}\text{div}(\mathcal{P}J_nu\otimes \mathcal{P}J_nu)-\Delta_{h}\mathcal{P}J_nu=\mathcal{P}J_n(\rho e_3),\\
\partial_t\rho +J_n\text{div}(J_nu J_n\rho)-\Delta_{h}J_n\rho=0,\\
(\rho,u)|_{t=0}=J_{n}(\rho_0,u_0).
\end{cases}
\end{equation}
The Cauchy-Lipschitz theorem entails that this system exists a unique maximal solution $(\rho_n,u_n)$ in $\mathcal{C}^{1}([0,T^{*}_{n}[;L^{2}_{n})$.
On the other hand, we observe that $J^{2}_n=J_n, \mathcal{P}^{2}=\mathcal{P}$ and $J_n\mathcal{P}=\mathcal{P}J_n$. It follows that $(\rho_n,\mathcal{P}u_n)$ and $(J_n\rho_n,J_n\mathcal{P}u_n)$ are also solutions. The uniqueness gives that $\mathcal{P}u_n=u_n, J_nu_n=u_n$ and $J_n\rho_n=\rho_n$. Therefore
\begin{equation}\label{approx}
\begin{cases}
\partial_tu_{n}+\mathcal{P}J_{n}\text{div}(u_{n}\otimes u_{n})-\Delta_{h}u_{n}=\mathcal{P}J_n(\rho_{n} e_3),\\
\partial_t\rho_{n} +J_n\text{div}(u_{n} \rho_{n})-\Delta_h\rho_n=0,\\
\text{div}u_{n}=0,\\
(\rho_{n},u_{n})|_{t=0}=J_{n}(\rho_0,u_0).
\end{cases}
\end{equation}
As the operators $J_n$ and $\mathcal{P}J_n$ are the orthogonal projectors for the $L^2$-inner product, the above formal calculations remain unchanged.
We will start with  the following stability results.
\begin{lemma}\label{lem5-1}
Let $u_0$ be a free divergence axisymmetric vector-field without swirl and $\rho_0$ an axisymmetric scalar function. Then
\begin{enumerate}[$(i)$]
\item for every $n\in\NN$, $u_{0,n}$ and $\rho_{0,n}$ are axisymmetric and $\textnormal{div}u_{0,n}=0.$
\item
If $u_0\in H^1$  is  such that
$({\textnormal{curl }u_0})/{r}\in L^2$
 and $\rho_{0}\in  H^{0,1}$. Then  there exists a constant $C$ independent of $n$ such that
$$
\|u_{0,n}\|_{H^1}\le \|u_0\|_{H^1},\quad
\big\| (\textnormal{curl }u_{0,n})/r \big\|_{L^2}\le C\big\| {(\textnormal{curl }u_0)}/{r} \big\|_{L^2},
 $$
 $$  \|\rho_{0, n}\|_{L^2} \leq \|\rho_{0}\|_{L^2}, \quad
  \|\rho_{0, n} \|_{H^{0,1}} \leq C \|\rho_{0} \|_{H^{0,1}}.
$$
\end{enumerate}
\end{lemma}
\begin{proof}
The proof of $\big\| (\textnormal{curl }u_{0,n})/r \big\|_{L^2}\le C\big\| {(\textnormal{curl }u_0)}/{r} \big\|_{L^2}$ is subtle, one can see \cite{rd} for more details. Other estimates can be proved by the standard methods.
\end{proof}

Now, we come back to the proof of the existence parts of Theorem \ref{thm1}. From Lemma \ref{lem5-1}, we observe that the initial
structure of axisymmetry is preserved for every $n$ and the involved
norms are uniformly controlled with respect to this parameter $n$. This ensures us to construct locally in time a unique solution
$(u_n,\rho_n)$ to the approximate system \eqref{approx}. On the other hand, we have seen in \mbox{Proposition \ref{Lipschitz}} that the
Lipschitz norm of the velocity keeps bounded in finite time. Therefore, this solution is globally defined.
By standard compactness arguments and Lions-Aubin Lemma we can
show that this family $(u_n,\rho_n)_{n\in\NN}$ converges to $(u,\rho)$ which satisfies in
turn our initial problem.  And the Fatou Lemma ensures $(u,\rho)\in\mathcal{X}$, where
\begin{align*}
\mathcal{X}:=&\big(L^\infty_{\rm loc}(\RR_{+};H^1)\cap L^2_{\rm loc}(\RR_{+};H^{2,1})\cap L^\infty_{\rm loc}(\RR_+;H^{1,1}\cap H^{0,2})\cap L^2_{\rm loc}(\RR_+;H^{2,1}\cap H^{1,2})\\&\cap L^1_{\rm loc}(\RR_{+};{\rm Lip})\big)
\times \big(L^\infty_{\rm loc}(\RR_{+}; H^{0,1})\cap L^2_{\rm loc}(\RR_{+}; H^{1,1})\big).
\end{align*}

It remains to prove the time continuity of the solution $(u,\rho)$. We only show that $u$ belongs to $ \mathcal{C}(\RR_{+};H^{1})$, the other terms can be treated the same way. First we show the continuity of $u$ in $H^{1}$. Indeed, we just need to show that $\omega\in \mathcal{C}(\RR_{+};L^{2})$. Let us recall the vorticity equation
\begin{equation*}
\partial_t \omega +u\cdot\nabla\omega-\Delta_{h}\omega
 =-\partial_{r}\rho e_{\theta}+\frac{u^r}{r}\omega.
\end{equation*}
It is easy to check that the source terms belong to $L^{2}_{\rm loc}(\RR_+;L^2)$. Using the fact $\nabla u\in L^{1}_{\rm loc}(\RR_+;{\rm Lip})$ and applying Proposition \ref{prop-con}, we get the desired result $\omega\in \mathcal{C}(\RR_{+};L^{2})$.

Next, let us turn to prove the uniqueness.
We assume that $(u_{i},\rho_{i})\in \mathcal{X}, 1\leq i\leq 2$ be  two solutions of the
 system \eqref{eq1.1} with the same initial \mbox{data $(u_0,\rho_0)$}.
   Then the difference $(\delta\rho,\delta u,\delta p)$ between two solutions $(\rho_1,u_1,p_1)$ and $(\rho_2,u_2,p_2)$ satisfies
\begin{equation}\label{diff}
\begin{cases}
\partial_t\delta u+\text{div}({u_{2}\otimes\delta u})-\Delta_{h}\delta u+\nabla\delta p=-\delta{u}\cdot\nabla u_{1}+\delta\rho e_{z},\\
\partial_t\delta\rho+\text{div}(u_{2}\delta\rho)-\Delta_h\delta\rho=-\delta u\cdot\nabla\rho_1.
\end{cases}
\end{equation}
Taking the $L^{2}$-inner product to the first equation of \eqref{diff} with $ u$, we obtain
\begin{equation}\label{diff-velo}
\begin{split}
\frac{1}{2}\frac{\rm d}{{\rm d} t}\norm{ \delta u(t)}^{2}_{L^{2}}+\norm{\nabla_{h} \delta u}^{2}_{L^{2}}=&-\int\delta u\nabla u_{1}\delta u {\rm d}x+\int\delta\rho e_{z}\delta u {\rm d}x\\
\leq&\norm{\nabla u_{1}}_{L^{\infty}}\norm{\delta u}^{2}_{L^{2}}+\norm{\delta \rho}_{L^2}\norm{\delta u}_{L^2}.
\end{split}
\end{equation}
On the other hand, by the same computation, we get
\begin{equation*}
\begin{split}
\frac{1}{2}\frac{\rm d}{{\rm d} t}\norm{ \delta \rho(t)}^{2}_{L^{2}}+\norm{\nabla_{h} \delta \rho}^{2}_{L^{2}}=-\int\delta u\nabla \rho_1\delta \rho {\rm d}x
=-\int(\delta u)^{r}\partial_{r} \rho_1\delta \rho {\rm d}x-\int(\delta u)^{z}\partial_{z} \rho_1\delta \rho {\rm d}x.
\end{split}
\end{equation*}
By Lemma \ref{lema.2} and the Young inequality,
\begin{equation*}
\begin{split}
\int(\delta u)^{r}\partial_{r} \rho_1\delta \rho {\rm d}x\leq& \norm{(\delta u)^{r}}^{\frac{1}{2}}_{L^{2}}\norm{\nabla_{h}(\delta u)^{r}}^{\frac{1}{2}}_{L^{2}}\norm{\partial_{r} \rho_1}^{\frac{1}{2}}_{L^{2}}\norm{\partial_{z}\partial_{r} \rho_1}^{\frac{1}{2}}_{L^{2}}
\norm{\delta\rho}^{\frac{1}{2}}_{L^{2}}\norm{\nabla_{h}\delta\rho}^{\frac{1}{2}}_{L^{2}}\\
\leq&C\norm{\partial_{r} \rho_{1}}_{L^{2}}\norm{\partial_{z}\partial_{r} \rho_{1}}_{L^{2}}
\norm{\delta u}_{L^{2}}\norm{\delta\rho}_{L^{2}}+\frac12\norm{\nabla_{h}\delta u}_{L^{2}}\norm{\nabla_{h}\delta\rho}_{L^{2}}.
\end{split}
\end{equation*}
Using Lemma \ref{lema.2} and $\text{div}\delta u=0$, we have
\begin{align*}
\int(\delta u)^{z}\partial_{z} \rho_1\delta \rho {\rm d}x\leq& \norm{(\delta u)^{z}}^{\frac{1}{2}}_{L^{2}}\norm{\partial_{z}(\delta u)^{z}}^{\frac{1}{2}}_{L^{2}}\norm{\partial_{z} \rho_1}^{\frac{1}{2}}_{L^{2}}\norm{\nabla_{h}\partial_{z} \rho_1}^{\frac{1}{2}}_{L^{2}}
\norm{\delta\rho}^{\frac{1}{2}}_{L^{2}}\norm{\nabla_{h}\delta\rho}^{\frac{1}{2}}_{L^{2}}\\
\leq& \norm{(\delta u)^{z}}^{\frac{1}{2}}_{L^{2}}\norm{\nabla_{h}(\delta u)}^{\frac{1}{2}}_{L^{2}}\norm{\partial_{z} \rho_1}^{\frac{1}{2}}_{L^{2}}\norm{\nabla_{h}\partial_{z} \rho_1}^{\frac{1}{2}}_{L^{2}}
\norm{\delta\rho}^{\frac{1}{2}}_{L^{2}}\norm{\nabla_{h}\delta\rho}^{\frac{1}{2}}_{L^{2}}\\
\leq&C\norm{\partial_{z} \rho_{1}}_{L^{2}}\norm{\nabla_{h}\partial_{z} \rho_{1}}_{L^{2}}
\norm{\delta u}_{L^{2}}\norm{\delta\rho}_{L^{2}}+\frac12\norm{\nabla_{h}\delta u}_{L^{2}}\norm{\nabla_{h}\delta\rho}_{L^{2}}.
\end{align*}
The combination of these estimates yield
\begin{equation*}
\begin{split}
&\frac{1}{2}\frac{\rm d}{{\rm d} t}\norm{ \delta \rho(t)}^{2}_{L^{2}}+\norm{\nabla_{h} \delta \rho}^{2}_{L^{2}}\\
\leq&C(\norm{\nabla_{h} \rho_{1}}_{L^{2}}+\norm{\partial_{z} \rho_{1}}_{L^{2}})\norm{\nabla_{h}\partial_{z} \rho_{1}}_{L^{2}}
\norm{\delta u}_{L^{2}}\norm{\delta\rho}_{L^{2}}+\norm{\nabla_{h}\delta u}_{L^{2}}\norm{\nabla_{h}\delta\rho}_{L^{2}}.
\end{split}
\end{equation*}
This together with \eqref{diff-velo} yields that
\begin{equation*}
\begin{split}
&\frac{1}{2}\frac{\rm d}{{\rm d} t}\left(\norm{ \delta \rho(t)}^{2}_{L^{2}}+\norm{ \delta u(t)}^{2}_{L^{2}}\right)+\norm{\nabla_{h} \delta \rho}^{2}_{L^{2}}+\norm{\nabla_{h} \delta u}^{2}_{L^{2}}\\
\leq&C\big(\norm{\nabla_{h} \rho_{1}}_{L^{2}}+\norm{\partial_{z} \rho_{1}}_{L^{2}}\big)\norm{\nabla_{h}\partial_{z} \rho_{1}}_{L^{2}}
\norm{\delta u}_{L^{2}}\norm{\delta\rho}_{L^{2}}+\norm{\nabla_{h}\delta u}_{L^{2}}\norm{\nabla_{h}\delta\rho}_{L^{2}}\\
&+\norm{\nabla u_1}_{L^{\infty}}\norm{\delta u}^{2}_{L^{2}}+\norm{\delta \rho}_{L^2}\norm{\delta u}_{L^2}.
\end{split}
\end{equation*}
Consequently,
\begin{equation*}
\frac{\rm d}{{\rm d} t}\left(\norm{ \delta \rho(t)}^{2}_{L^{2}}+\norm{ \delta u(t)}^{2}_{L^{2}}\right)\le C F(t) \left(\norm{ \delta \rho(t)}^{2}_{L^{2}}+\norm{ \delta u(t)}^{2}_{L^{2}}\right),
\end{equation*}
where
$$F(t)=(\norm{\nabla_{h} \rho_{1}}_{L^{2}}+\norm{\partial_{z} \rho_{1}}_{L^{2}})\norm{\nabla_{h}\partial_{z} \rho_{1}}_{L^{2}}+\norm{\nabla u_1}_{L^{\infty}}+1.
$$
By Proposition \ref{Prop-Energy} and Proposition \ref{vertical}, we know that $F(t)$ is integrable.
Therefore, we obtain the uniqueness by using the Gronwall inequality.
\section{Proof of Theorem \ref{lose-global}}\label{sectionproof-2}

In this section, we intend to prove the global existence and the uniqueness of Theorem \ref{lose-global} for another class of initial data.
\begin{proposition}\label{log}
Assume that $u_0\in H^{1},$ with $\frac{\omega_0}{r}\in L^2$ and $\omega_{0}\in L^{\infty}$. Let $\rho_0\in H^{0,1}$. Then any smooth axisymmetric solution $(u,\rho)$ of  \eqref{eq1.1} without swirl satisfies
\begin{equation*}
\norm{\nabla u(t)}_{L}\leq Ce^{\exp{Ct}}\big(\norm{\omega_0}_{L^2\cap L^\infty}+1\big).
\end{equation*}
Here constant $C$ depends on the initial data.
\end{proposition}
\begin{proof}
Multiplying the vorticity equation \eqref{tourbillon} with $|\omega_\theta|^{p-2}\omega_\theta$ and performing integration in space, we get
\begin{equation}
\begin{split}
&\frac1p\frac{\rm d}{{\rm d}t}\int |\omega_\theta|^p{\rm d}x+(p-1)\int|\nabla_h\omega_\theta|^{2}|\omega_\theta|^{p-2}{\rm d}x+\int|\omega_\theta|^{p-2}\frac{\omega^{2}_\theta}{r^{2}}{\rm d}x\\
=&\int\frac{u^r}{r} |\omega_\theta|^p{\rm d}x-\int\partial_{r}\rho|\omega_\theta|^{p-2}\omega_\theta {\rm d}x.
\end{split}
\end{equation}
We consider the case $p\geq 4$.
For the first term in the last line, we deuce by the H\"older inequality that
\begin{equation*}
\int\frac{u^r}{r} |\omega_\theta|^p{\rm d}x\leq \Big\|\frac{u^{r}}{r}\Big\|_{L^{\infty}}\norm{\omega_\theta}^{p}_{L^p}.
\end{equation*}
 Since $\norm{u}_{L^{p-2}}\leq\norm{u}^{\alpha}_{L^{2}}\norm{u}^{1-\alpha}_{L^{p}}
$ with $\alpha=\frac{4}{(p-2)^2}$ and
\begin{equation*}
\int|\omega_\theta|^{p-4}|\nabla_{h}\omega_\theta|^{2}{\rm d}x=\int|\omega_\theta |^{p-4}|\nabla_{h}\omega_\theta|^{\frac{2(p-4)}{p-2}}|\nabla_{h}\omega_\theta|^{\frac{4}{p-2}}{\rm d}x
\leq\norm{\nabla_{h}\omega_\theta}^{\frac{4}{p-2}}_{L^{2}}\big\||\omega_\theta|^{\frac{p-2}{2}}\nabla_h\omega_\theta\big\|^{\frac{2(p-4)}{p-2}}_{L^{2}}.
\end{equation*}
By using Lemma \ref{lema.2} and some based inequalities, the second term can be bounded as follows
\begin{align*}
&\int\partial_{r}\rho|\omega_\theta|^{p-2}\omega_\theta {\rm d}x\\
\leq&\norm{\partial_r\rho}^{\frac{1}{2}}_{L^{2}}\norm{\partial^2_{rz}\rho}^{\frac{1}{2}}_{L^2}\norm{\omega^{\frac{p-2}{2}}_\theta}^{\frac{1}{2}}_{L^2}
\norm{\nabla_h\big(\omega^{\frac{p-2}{2}}_\theta\big)}^{\frac12}_{L^2}
\norm{\omega^{\frac{p}{2}}_\theta}^{\frac{1}{2}}_{L^2}\norm{\nabla_h\big(\omega^{\frac{p}{2}}_\theta\big)}^{\frac{1}{2}}_{L^2}\\
\leq&C\sqrt{p}\norm{\partial_r\rho}^{\frac{1}{2}}_{L^{2}}\norm{\partial^2_{rz}\rho}^{\frac{1}{2}}_{L^2}\norm{\omega_\theta}^{\frac{p-2}{4}}_{L^{p-2}}
\norm{\omega^{\frac{p}{2}-2}_\theta\nabla_h\omega_\theta}^{\frac{1}{2}}_{L^2}
\norm{\omega_\theta}^{\frac{p}{4}}_{L^p}\norm{\nabla_h\big(\omega^{\frac{p}{2}}_\theta\big)}^{\frac{1}{2}}_{L^2}\\
\leq&Cp^{\frac{1}{p-2}}\norm{\partial_r\rho}^{\frac{1}{2}}_{L^{2}}\norm{\partial^2_{rz}\rho}^{\frac{1}{2}}_{L^2}\norm{\omega_\theta}^{\frac{1}{p-2}}_{L^2}
\norm{\omega_\theta}^{\frac{p(p-4)}{4(p-2)}}_{L^{p}}
\norm{\omega_\theta}^{\frac{p}{4}}_{L^p}\norm{\nabla_h\omega_\theta}^{\frac{1}{p-2}}_{L^{2}}
\norm{\nabla_h\big(\omega^{\frac{p}{2}}_\theta\big)}^{\frac12+\frac{(p-4)}{2(p-2)}}_{L^2}\\
\leq&Cp^{\frac{2}{p-1}}\norm{\partial_r\rho}^{\frac{p-2}{p-1}}_{L^{2}}\norm{\partial^2_{rz}\rho}^{\frac{p-2}{p-1}}_{L^2}
\norm{\nabla_h\omega_\theta}^{\frac2{p-1}}_{L^2}
\norm{\omega_\theta}^{\frac{2}{p-1}}_{L^2}
\norm{\omega_\theta}^{p-2-\frac{2}{p-1}}_{L^p}+\frac14\norm{\nabla_h\big(\omega^{\frac{p}{2}}_\theta\big)}^{2}_{L^2}
\end{align*}
Without loss of generality, we assume that $\norm{\omega_\theta}_{L^{p}}\geq 1$, thus
\begin{equation*}
\frac{\rm d}{{\rm d}t}\norm{\omega_\theta}^{2}_{L^{p}}\leq C \Big\|\frac{u^{r}}{r}\Big\|_{L^{\infty}}\norm{\omega_\theta}^{2}_{L^p}+Cp^{\frac{2}{p-1}}F(t)\leq C \Big\|\frac{u^{r}}{r}\Big\|_{L^{\infty}}\norm{\omega_\theta}^{2}_{L^p}+4CF(t),
\end{equation*}
where $F(t):=\norm{\partial_r\rho}^{\frac{p-2}{p-1}}_{L^{2}}\norm{\partial^2_{rz}\rho}^{\frac{p-2}{p-1}}_{L^2}
\norm{\nabla_h\omega_\theta}^{\frac2{p-1}}_{L^2}
\norm{\omega_\theta}^{\frac{2}{p-1}}_{L^2}$.
According  to  Proposition \ref{Strong} and Proposition \ref{vertical}, we know that $F(t)$ is integrable.
Therefore, by the Gronwall inequality and the relation $\norm{\omega}_{L^{p}}=\norm{\omega_\theta}_{L^{p}}$,
we obtain that 
\begin{equation*}
\norm{\omega}^{2}_{L^{p}}\leq Ce^{\exp{Ct}}\Big(\norm{\omega_0}^{2}_{L^{p}}+\int^{t}_{0}F(\tau){\rm d}\tau\Big).
\end{equation*}
This together with the second estimate of Proposition \ref{Strong} yields
\begin{equation*}
\norm{\omega}^{2}_{L^{p}}\leq Ce^{\exp{Ct}}\quad {\rm for}\quad 2\leq p<\infty.
\end{equation*}
Since  $\norm{\nabla u}_{L^{p}}\leq C\frac{p^{2}}{p-1}\norm{\omega}_{L^{p}}$ \mbox{(see \cite{che1}. Chap-3)}. So, we finally obtain that
$
\norm{\nabla u}_{L}\leq Ce^{\exp{Ct}}.$
This ends the proof.
\end{proof}

Let us first focus on the existence part of Theorem \ref{lose-global}. Let
\begin{align*}
\mathcal{Y}:=&\big(L^\infty_{\rm loc}(\RR_+;H^1)\cap L^2_{\rm loc}(\RR_+;H^{1,1})\cap L^\infty_{\rm loc}(\RR_+;H^{1,1}\cap H^{0,2})\cap L^2_{\rm loc}(\RR_+;H^{2,1}\cap H^{1,2})\\&\cap L^1_{\rm loc}(\RR_+;{\rm L})\big)
\times\big( L^\infty_{\rm loc} (\RR_+;H^{0,1})\cap L^2_{\rm loc}(\RR_+;H^{1,1})\big).
\end{align*}To prove existence, we smooth out the initial data $(u_{0},\rho_{0})$ so as to obtain a sequence $(u_{0,n},\rho_{0,n})_{n\in\NN}$ of smooth functions which converges
to $(u_0,\rho_0)$. From Lemma \ref{lem5-1}, it is clear that the initial
structure of axisymmetry is preserved for every $n$ . By preceding argument, it is easy to check that $(u_n,\rho_n)\in\mathcal{Y}$.
Combining this with the equations \eqref{eq1.1}, one may conclude that
$\partial_{t}\rho_n\in L_{\text{loc}}^{2}(\RR_+;H^{-1})$ and $\partial_{t}u_n\in L_{\text{loc}}^{2}(\RR_+;L^2)$. On the other hand, we know that $L^{2}\hookrightarrow H^{-1}$ and $H^{1}\hookrightarrow L^{2}$ are locally compact.
Therefore, by the classical Aubin-Lions argument and Cantor's diagonal process, we can deduce that, up to extraction,
family $(u_n,\rho_n)_{n\in\NN}$ has a limit $(u,\rho)$ satisfying the equations \eqref{eq1.1} and that $(u,\rho)\in\mathcal{Y}$.
The same arguments as used in \mbox{Proposition \ref{prop-con}} allows us to show that the time continuity of $(u,\rho)$ in low norms and the weak time continuity.
In addition, in a similar way as used in \mbox{Theorem 5 of \cite{Diperna}}, we can conclude that $\rho\in \mathcal{C}_{b}(\RR_+;L^2)$.

Now let us turn to the prove the uniqueness. We assume that $(u_{i},\rho_{i})\in \mathcal{Y}, 1\leq i\leq 2$ be  two solutions of the
 system \eqref{eq1.1} with the same initial \mbox{data $(u_0,\rho_0)$}.
One can write \eqref{diff}
\begin{equation}
\begin{cases}
\partial_t\delta u+\text{div}({u_{2}\otimes\delta u})+\text{div}({\delta u\otimes u_{1}})-\Delta_{h}\delta u+\nabla\delta p=\delta\rho e_{z},\\
\partial_t\delta\rho+\text{div}(u_{2}\delta\rho)-\Delta_h\delta\rho=-\text{div}(\delta u\rho_1).
\end{cases}
\end{equation}
For the sake of convenience, let $(\alpha,\beta,\gamma)$ such that $\frac12<\alpha<\beta<\gamma\leq 1.$ Note that 
$$\frac{\norm{S_{q}u}_{L^{\infty}}}{q}\leq 2^{\frac{3}{q}}\frac{\norm{S_{q}u}_{L^{q}}}{q}\leq 2^{\frac{3}{2}}\norm{u}_{L}.
$$
By the same argument as in Proposition \ref{lostest} with the vector-field $u_2$, then there exists $T_{1}$ such that
\begin{equation*}
\norm{\delta\rho}_{H^{\beta-1}}\leq C\int_{0}^{t}\norm{\text{div}(\delta u\rho_1)}_{H^{\gamma-1}}{\rm d}\tau\quad\text{for all } t\in [0,T_{1}].
\end{equation*}
The term on the right side can be bounded as follows. By virtue of the Bony decomposition:
\begin{equation}\label{para}
\text{div}(\delta u\rho_1)=\text{div}\big(T_{\delta u}\rho_1 +R(\delta u,\rho_1)\big)+\sum^{2}_{i=1}T_{\partial_i\rho_1}\delta u^i,
\end{equation}
where we have used the condition $\text{div}\delta u=0$.

From standard continuity results for operators $T$ and $R$ \mbox{(see for example \cite{BCD11})}, we have
\begin{equation*}
\norm{T_{\delta u}\rho_1+R(\delta u,\rho_1)}_{H^{\gamma}}\leq C\norm{\delta u}_{L^\infty}\norm{\rho_1}_{H^{\gamma}}.
\end{equation*}
As for the last term, since $\gamma-1<0$, we infer that
$$\norm{T_{\partial_i\rho_1}\delta u^i}_{H^{\gamma-1}}\leq C\norm{\nabla\rho_1}_{H^{\gamma-1}}\norm{\delta u}_{L^{\infty}}.
$$
We eventually get
\begin{equation}\label{de-p}
\norm{\delta\rho}_{L_{t}^{\infty}(H^{\beta-1})}\leq C\norm{\rho_1}_{L_{t}^{2}(H^{\gamma})}\norm{\delta u}_{L_{t}^{2}L^{\infty}}.
\end{equation}
Now we turn to bound the term $\delta u$. By using Proposition \ref{lostest}, there exists $T_{2}$ such that for all $t\in [0,T_{2}],$
\begin{equation*}
\norm{\delta u}_{L^{\infty}_{t}(H^\alpha)}+\norm{\nabla_h\delta u}_{L^{2}_{t}(H^\alpha)}\leq C\big(\norm{\delta\rho}_{L^{2}_{t}(H^\beta)}+\norm{\delta u\cdot\nabla u_1}_{L^{2}_{t}(H^\beta)}\big)
\end{equation*}
for some constant $C$ depending only on $\alpha,\beta$ and $u_2.$ Using again the Bony decomposition and arguing exactly as for proving \eqref{de-p}, we get
$$\norm{\delta u\cdot \nabla u_1}_{H^{\beta-1}}\leq C\norm{\delta u}_{L^{\infty}}\norm{u_1}_{H^{\beta}}.
$$
Therefore, given that $u_1\in L_{{\rm loc}}^{\infty}(\RR;H^{\beta})$,
\begin{equation}
\norm{\delta u}_{L^{\infty}_{t}(H^{\alpha})}+\norm{\nabla _h \delta u}_{L^{2}_{t}(H^{\alpha})}\leq C(\norm{\delta\rho}_{L^{2}_{t}(H^{\beta-1})}+
\norm{\delta u}_{L^{2}_{t}(L^{\infty})}).
\end{equation}
Next, our task is to show that $\norm{\delta u}_{L_{t}^{2}L^{\infty}}$
 may be bounded in terms of $\norm{\delta u}_{L^{\infty}_{t}H^{\alpha}}$ and of $\norm{\nabla_h\delta u}_{L^{2}_{t}H^{\alpha}}$.

According to the assumption $\alpha\in ]\frac12,1]$, we have (see the proof in \mbox{Appendix \ref{appendix})}
\begin{equation}\label{appen-l}\norm{\delta u}_{L^{\infty}(\RR^{3})}\leq C\norm{\delta u}^{\alpha-\frac12}_{H^\alpha(\RR^{3})}\norm{\nabla_h\delta u}^{\frac32-\alpha}_{H^\alpha(\RR^{3})}.
\end{equation}
Combining these estimates, we can deduce that for some constant $C$ depending only on $T=\min\{T_1,T_2\}$ and on the norms of $(\rho_1,u_1)$ and $(\rho_2,u_2)$, we have
$$\norm{\delta \rho}_{L^{\infty}_{t}H^{\beta-1}}\leq Ct^{\frac\alpha2-\frac14}\delta U(t),
\quad \delta U(t)\leq C\big(t^{\frac12}\norm{\delta\rho}_{L^{\infty}_{t}H^{\beta-1}}+t^{\frac\alpha2-\frac14}\delta U(t)\big)
$$
with
$$\delta U(t):=\norm{\delta u}_{L^{\infty}_{t}H^{\alpha}}+\norm{\nabla_h\delta u}_{L^{2}_{t}H^{\alpha}}.
$$
It follows that $\delta u\equiv 0$ (and thus $\delta\rho\equiv 0$) on a suitably small time interval.
Finally, let us notice that our assumptions on the solutions ensure that $\delta\rho\in \mathcal{C}([0,T];H^{\beta-1})$ and $\delta u\in \mathcal{C}([0,T];H^{\alpha})$. Using a classical connectivity argument, it is now easy to get the uniqueness on the whole interval $[0,\infty[$.

\appendix

\section{Appendix}
\label{appendix}
\setcounter{section}{6}\setcounter{equation}{0}

In this section, we first give some useful inequalities which have been used throughout the paper.

\begin{lemma}\label{sharp}
There exists a constants $C$ such that
\begin{equation}
\norm{u}_{L^\infty(\RR^{3})}\leq C\norm{\nabla u}^{\frac12}_{L^{2}(\RR^{3})}\norm{\nabla_{h}\nabla u}^{\frac12}_{L^{2}(\RR^{3})}.
\end{equation}
\end{lemma}
\begin{proof}
By using the interpolation theorem, we get
\begin{equation*}
\norm{u(x_h,\cdot)}_{L^{\infty}(\RR_{v})}\leq C\|u(x_h,\cdot)\|^{\frac12}_{L^{6}(\RR_{v})}\|\Lambda^{\frac23}_{v}u(x_h,\cdot)\|^{\frac12}_{L^{2}(\RR_{v})}.
\end{equation*}
This together with the Minkowski inequality and the embedding  theorem gives
\begin{equation}\label{sharp-1}
\begin{split}
\norm{u}_{L^\infty(\RR^{3})}\leq& \norm{\|u\|_{L^{\infty}(\RR_{v})}}_{L^{\infty}(\mathbb{R}_{h}^{2})}\\\leq & C\big\|\|u\|^{\frac12}_{L^{6}(\RR_{v})}\|\Lambda^{\frac23}_{v}u\|^{\frac12}_{L^{2}(\RR_{v})}\big\|_{L^{\infty}(\mathbb{R}_{h}^{2})}\\
\leq &\big\|\|\Lambda^{\frac13}_{v}u\|_{L^{\infty}(\RR_{h}^{2})}\big\|^{\frac12}_{L^{2}(\RR_{v})}
\big\|\|\Lambda_{v}^{\frac23}u\|_{L^{\infty}(\RR_{h}^{2})}\big\|^{\frac12}_{L^{2}(\RR_{v})}.
\end{split}
\end{equation}
On the other hand, using again the interpolation theorem and the embedding theorem, we have
\begin{equation}\label{sharp-2}
\begin{split}
\|\Lambda^{\frac13}_{v}u(\cdot,z)\|_{L^{\infty}(\RR_{h}^{2})}\leq & C\|\Lambda^{\frac13}_{v}u(\cdot,z)\|^{\frac23}_{L^{6}(\RR_{h}^{2})}\|\Lambda^{\frac53}_{h}\Lambda^{\frac13}_{v}u(\cdot,z)\|^{\frac13}_{L^{2}(\RR_{h}^{2})}\\
\leq& C\|\Lambda^{\frac23}_{h}\Lambda^{\frac13}_{v}u(\cdot,z)\|^{\frac23}_{L^{2}(\RR_{h}^{2})}\|\Lambda^{\frac53}_{h}
\Lambda^{\frac13}_{v}u(\cdot,z)\|^{\frac13}_{L^{2}(\RR_{h}^{2})}
\end{split}
\end{equation}
and
\begin{equation}\label{sharp-3}
\begin{split}
\|\Lambda^{\frac23}_{v}u(\cdot,z)\|_{L^{\infty}(\RR_{h}^{2})}\leq & C\|\Lambda^{\frac23}_{v}u(\cdot,z)\|^{\frac13}_{L^{3}(\RR_{h}^{2})}\|\Lambda^{\frac43}_{h}\Lambda^{\frac23}_{v}u(\cdot,z)\|^{\frac23}_{L^{2}(\RR_{h}^{2})}\\
\leq& C\|\Lambda^{\frac13}_{h}\Lambda^{\frac23}_{v}u(\cdot,z)\|^{\frac13}_{L^{2}(\RR_{h}^{2})}\|\Lambda^{\frac43}_{h}
\Lambda^{\frac23}_{v}u(\cdot,z)\|^{\frac23}_{L^{2}(\RR_{h}^{2})}.
\end{split}
\end{equation}
Inserting \eqref{sharp-2} and  \eqref{sharp-3} into \eqref{sharp-1}, and using the H\"older inequality, we get
\begin{align*}
\norm{u}_{L^\infty(\RR^{3})}\leq&\big\|\|\Lambda^{\frac23}_{h}\Lambda^{\frac13}_{v}u\|^{\frac23}_{L^{2}(\RR_{h}^{2})}\|\Lambda^{\frac53}_{h}
\Lambda^{\frac13}_{v}u\|^{\frac13}_{L^{2}(\RR_{h}^{2})}\big\|^{\frac12}_{L^{2}(\RR_{v})}
\big\|\|\Lambda^{\frac13}_{h}\Lambda^{\frac23}_{v}u\|^{\frac13}_{L^{2}(\RR_{h}^{2})}\|\Lambda^{\frac43}_{h}
\Lambda^{\frac23}_{v}u\|^{\frac23}_{L^{2}(\RR_{h}^{2})}\big\|^{\frac12}_{L^{2}(\RR_{v})}\\
\leq&\|\Lambda^{\frac23}_{h}\Lambda^{\frac13}_{v}u\|^{\frac13}_{L^{2}(\RR^{3})}\|\Lambda^{\frac53}_{h}
\Lambda^{\frac13}_{v}u\|^{\frac16}_{L^{2}(\RR^{3})}
\|\Lambda^{\frac13}_{h}\Lambda^{\frac23}_{v}u\|^{\frac16}_{L^{2}(\RR^{3})}\|\Lambda^{\frac43}_{h}
\Lambda^{\frac23}_{v}u\|^{\frac13}_{L^{2}(\RR^{3})}\\
\leq &C\norm{\nabla u}^{\frac12}_{L^{2}(\RR^{3})}\norm{\nabla_{h}\nabla u}^{\frac12}_{L^{2}(\RR^{3})}.
\end{align*}
This completes the proof.
\end{proof}
\begin{lemma}\label{lema.1}
Let $q\in]2,\infty[$, there holds that
\begin{equation}\label{a.1}
\begin{split}
\int_{\mathbb{R}^{3}}fgh{\rm d}x_{1}{\rm d}x_{2}{\rm d}x_{3}
\leq C\norm{f}^{\frac{q-1}{q}}_{L^{2(q-1)}}\norm{\partial_{x_{1}}f}^{\frac{1}{q}}_{L^{2}}\norm{g}^{\frac{q-2}{q}}_{L^{2}}\norm{\partial_{x_2}g}^{\frac{1}{q}}_{L^{2}}
\norm{\partial_{x_3}g}^{\frac{1}{q}}_{L^{2}}\norm{h}_{L^{2}}.
\end{split}
\end{equation}
In particular, if we take $q=4$ in \eqref{a.1}, we have
\begin{equation}\label{a.11}
\begin{split}
\int_{\mathbb{R}^{3}}fgh{\rm d}x_{1}{\rm d}x_{2}{\rm d}x_{3}
\leq C\norm{f}^{\frac{3}{4}}_{L^{6}}\norm{\partial_{x_{3}}f}^{\frac{1}{4}}_{L^{2}}\norm{g}^{\frac{1}{2}}_{L^{2}}\norm{\nabla_{h}g}^{\frac{1}{2}}_{L^{2}}\norm{h}_{L^{2}}.
\end{split}
\end{equation}
\end{lemma}
\begin{proof}
We only just to show the inequality for functions $f,g,h\in C^{\infty}_{0}(\mathbb{R}^{3})$ and then pass to the limit by virtue of the density argument.

Using some basic inequalities, we have
\begin{align*}
&\int_{\mathbb{R}^{3}}fgh{\rm d}x_{1}{\rm d}x_{2}{\rm d}x_{3}\\
\leq&C\int_{\mathbb{R}^{2}}\left[\max_{x_{1}}|f|\left(\int_{\mathbb{R}}g^{2}{\rm d}x_{1}\right)^{\frac{1}{2}}\left(\int_{\mathbb{R}}h^{2}{\rm d}x_{1}\right)^{\frac{1}{2}}\right]{\rm d}x_{2}{\rm d}x_{3}\\
\leq&C\left[\int_{\mathbb{R}^{2}}\max_{x_{1}}|f|^{q}{\rm d}x_{2}{\rm d}x_{3}\right]^{\frac{1}{q}}\left[\int_{\mathbb{R}^{2}}\left(\int_{\mathbb{R}}g^{2}{\rm d}x_{1}\right)^{\frac{q}{q-2}}
{\rm d}x_{2}{\rm d}x_{3}\right]^{\frac{q-2}{2q}}\left(\int_{\mathbb{R}^{3}}h^{2}{\rm d}x_{1}{\rm d}x_{2}{\rm d}x_{3}\right)^{\frac{1}{2}}\\
\leq&C\left[\int_{\mathbb{R}^{2}}\int_{\mathbb{R}}|f|^{q-1}|\partial_{x_{1}f}|{\rm d}x_{1}{\rm d}x_{2}{\rm d}x_{3}\right]^{\frac{1}{q}}\left[\int_{\mathbb{R}^{2}}
\left(\int_{\mathbb{R}}g^{2}{\rm d}x_{1}\right)^{\frac{q}{q-2}}
{\rm d}x_{2}{\rm d}x_{3}\right]^{\frac{q-2}{2q}}\norm{h}_{L^{2}}\\
\leq& C\norm{f}^{\frac{q-1}{q}}_{L^{2(q-1)}}\norm{\partial_{x_{1}}f}^{\frac{1}{q}}_{L^{2}}\norm{g}^{\frac{q-2}{q}}_{L^{2}}\norm{\partial_{x_2}g}^{\frac{1}{q}}_{L^{2}}
\norm{\partial_{x_3}g}^{\frac{1}{q}}_{L^{2}}\norm{h}_{L^{2}}.
\end{align*}
Indeed, by imbedding theorem, H$\rm\ddot{o}$lder's inequality and Plancherel theorem, we obtain that
\begin{align*}
\big\|\norm{g}_{L_{2,3}^{\frac{2q}{q-2}}(\mathbb{R}^{2})}\big\|_{L_{1}^{2}(\mathbb{R})}\leq &C\big\|\big\|\|\Lambda^{\frac{1}{q}}_{2}g\|_{L_{2}^{2}(\mathbb{R}^{1})}\big\|_{L_{3}^{\frac{2q}{q-2}}(\mathbb{R}^{1})}\big\|_{L_{1}^{2}(\mathbb{R})}
\leq C\big\|\big\|\|\Lambda^{\frac{1}{q}}_{2}g\|_{L_{3}^{\frac{2q}{q-2}}(\mathbb{R}^{1})}\big\|_{L_{2}^{2}(\mathbb{R}^{1})}\big\|_{L_{1}^{2}(\mathbb{R})}\\
\leq &C\big\|\big\|\|\Lambda^{\frac{1}{q}}_{3}\Lambda^{\frac{1}{q}}_{2}g\|_{L_{3}^{2}(\mathbb{R}^{1})}\big\|_{L_{2}^{2}(\mathbb{R}^{1})}\big\|_{L_{1}^{2}(\mathbb{R})}=
C\big\|\Lambda^{\frac{1}{q}}_{3}\Lambda^{\frac{1}{q}}_{2}g\big\|_{L^{2}}\\
=&C\Big(\int_{\mathbb{R}^{3}}|\xi_{2}|^{\frac{2}{q}}|\xi_{3}|^{\frac{2}{q}}\hat{g}^{2}(\xi){\rm d}\xi_{1} {\rm d}\xi_{2} {\rm d}\xi_{3}\Big)^{\frac{1}{2}}\leq \norm{\hat{g}}^{\frac{q-2}{q}}_{L^{2}}\norm{\xi_{2}\hat{g}}^{\frac{1}{q}}_{L^{2}}\norm{\xi_{3}\hat{g}}^{\frac{1}{q}}_{L^{2}}\\
\leq& C\norm{g}^{\frac{q-2}{q}}_{L^{2}}\norm{\partial_{x_2}g}^{\frac{1}{q}}_{L^{2}}
\norm{\partial_{x_3}g}^{\frac{1}{q}}_{L^{2}}.
\end{align*}
This completes the proof.
\end{proof}

\begin{lemma}\label{lema.2}\cite{Cp}
A constant $C$ exists such that
\begin{equation}\label{a.5}
\begin{split}
\int_{\mathbb{R}^{3}}fgh{\rm d}x_{1}{\rm d}x_{2}{\rm d}x_{3}
\leq C\norm{f}^{\frac{1}{2}}_{L^{2}}\norm{\partial_{x_{3}}f}^{\frac{1}{2}}_{L^{2}}
\norm{g}^{\frac{1}{2}}_{L^{2}}\norm{\nabla_{h}g}^{\frac{1}{2}}_{L^{2}}\norm{h}^{\frac{1}{2}}_{L^{2}}\norm{\nabla_{h}h}^{\frac{1}{2}}_{L^{2}}.
\end{split}
\end{equation}
\end{lemma}
\begin{proof}[Proof of Inequality \eqref{appen-l}]
As for $\alpha\in]\frac12,\frac32[$, by the interpolation theorem and the embedding theorem, we get that for any $z\in\RR$
\begin{align}\label{ad-1}
\norm{u(\cdot,z)}_{L_{h}^{\infty}(\RR^{2})}\leq& C\norm{u(\cdot,z)}^{\alpha-\frac{1}{2}}_{L_{h}^{\frac{4}{3-2\alpha}}(\RR^{2})}\norm{\Lambda^{\alpha+\frac12}_{h}u(\cdot,z)}^{\alpha+\frac{1}{2}}_{L_{h}^{2}(\RR^{2})}\nonumber\\
\leq &C\norm{\Lambda^{\alpha-\frac12}_{h}u(\cdot,z)}^{\alpha-\frac{1}{2}}_{L_{h}^{2}(\RR^{2})}\norm{\Lambda^{\alpha+\frac12}_{h}u(\cdot,z)}^{\frac{3}{2}-\alpha}_{L_{h}^{2}(\RR^{2})}\\
\leq &C\norm{u(\cdot,z)}^{\alpha-\frac{1}{2}}_{H^{\alpha-\frac12}(\RR^{2})}\norm{\nabla_h u(\cdot,z)}^{\frac{3}{2}-\alpha}_{H^{\alpha-\frac12}(\RR^{2})}.\nonumber
\end{align}
On the other hand, by the trace theorem that $H^{\alpha}(\RR^{3})\hookrightarrow L^{\infty}(\RR_{v};H^{\alpha-\frac12}(\RR_{h}^{2}))$,
\begin{equation*}
\norm{u}_{L_{v}^{\infty}(H^{\alpha-\frac12})}\leq C\norm{u}_{(H^{\alpha})}\quad{\rm and }\quad\norm{\nabla_hu}_{L_{v}^{\infty}(H^{\alpha-\frac12})}\leq C\norm{\nabla_hu}_{(H^{\alpha})}.
\end{equation*}
Inserting these inequalities in \eqref{ad-1}, we get the desired result \eqref{appen-l}.
\end{proof}

For the sake of completeness, we  give an existence  result for the anisotropic equations with
a convection term, which is similar to the case of the transport equation in  \cite{BCD11,CM-1}.

\begin{proposition}\label{prop-con}
Let $s\geq-1$, $1\leq p< \infty$ and $1\leq r\leq \infty$.  Assume that $f_0\in B^{s}_{p,r},g\in L^{1}([0,T];B^{s}_{p,r})$, and that $u$ be a divergence free vector-field
satisfying $u\in L^{\sigma}([0,T];B^{-m}_{\infty,\infty})$ for some $\sigma>1$ and $m>0$, and $\nabla u\in  L^{1}([0,T];L^\infty)$.
Then the following equations
\begin{equation}\label{continuce}
\begin{cases}
\partial_tf+u\cdot\nabla f-\Delta_{h} f=g,\\
f|_{t=0}=f_{0}
\end{cases}
\end{equation}admits a unique solution $f$ in
\begin{itemize}
  \item the space $\mathcal{C}([0,T];B^{s}_{p,r})$, if $r<\infty$,
  \item the space $\big(\cap_{s'<s}\mathcal{C}([0,T];B^{s'}_{p,\infty})\big)\cap\mathcal{C}_{w}([0,T];B^{s}_{p,\infty})$, if $r=\infty$.
\end{itemize}
Moreover, we have
\begin{equation}\label{guji}
\norm{f(t)}_{B^{s}_{p,r}}\leq e^{C\int_{0}^{t}V(\tau){\rm d}\tau}\Big(\norm{f_0}_{B^{s}_{p,r}}+\int^{t}_{0}e^{-C\int_{0}^{\tau}V(s){\rm d}s}\norm{g(\tau)}_{B^{s}_{p,r}}{\rm d}\tau\Big).
\end{equation}
Here $V(t):=\|\nabla u(t)\|_{ L^{\infty}}$.
\end{proposition}
\begin{proof}
Without loss of generality, we assume that $u$ and $g$ are defined on $\RR\times\RR^{n}$.
 We first construct the approximate solutions $f_n$ of \eqref{continuce} as follows.
\begin{equation}\label{continuce-approx}
\begin{cases}
\partial_tf_{n}+u_{n}\cdot\nabla f_{n}-\Delta_{h} f_{n}=g_{n},\\
u_{n}:=\varphi_{n}\ast_{t}S_{n}u,\quad g_{n}:=\varphi_{n}\ast_{t}S_{n}g,\\
f_{n}|_{t=0}=f_{0,n}:=S_{n}f_{0},
\end{cases}
\end{equation}
where, $\varphi_{n}$ denotes  a family of mollifiers with respect to $t$. Thanks to the properties of mollifier and the operator $S_j$, it is clear that
$f_{0,n}\in B^{\infty}_{p,r}$ and $u_{n},g_{n}\in \mathcal{C}([0,T];B^{\infty}_{p,r})$ with $B^{\infty}_{p,r}:=\cap_{s\in\RR}B^{s}_{p,r}$. Moreover,
$(f_{0,n})_{n\in\NN}$ is bounded in $B_{p,r}^{s}$, $(g_{n})_{n\in\NN}$ is bounded in $L^{1}([0,T];B_{p,r}^{s})$, $(u_{n})_{n\in\NN}$ is bounded in $L^{\sigma}([0,T];B_{\infty,\infty}^{-m})$, and $(\nabla u_{n})_{n\in\NN}$ is bounded in $L^{1}([0,T]; L^{\infty})$.

Applying $\Delta_{k}$ to  \eqref{continuce-approx}, we have
\begin{equation*}
\begin{cases}
\partial_t\Delta_{k}f_{n}+u_{n}\cdot\nabla \Delta_{k}f_{n}-\Delta_{h} \Delta_{k}f_{n}=\Delta_{k}g_{n}+R^{k}_{n},\\
\Delta_{k}f_{n}|_{t=0}=\Delta_{k}f_{0,n},
\end{cases}
\end{equation*}
where $R^{k}_{n}:=u_n\cdot\nabla\Delta_{k}f_n-\Delta_{k}(u\cdot\nabla f_{n})$. Note that
$$-\int \Delta_{h}(\Delta_{k}f_n)|\Delta_{k}f_n|^{p-2}\Delta_{k}f_n{\rm d}x\geq 0,
$$
it is easy to conclude that
\begin{equation}\label{frequency}\norm{\Delta_{k}f_n(t)}_{L^{p}}\leq\norm{\Delta_{k}f_0}_{L^{p}}+\int_{0}^{t}\norm{\Delta_{k}g_n(\tau)}_{L^{p}}{\rm d}\tau+\int_{0}^{t}\norm{R^{k}_{n}(\tau)}_{L^{p}}{\rm d}\tau.
\end{equation}
This together with the commutator estimate (see e.g. \cite{CM-1}, Chap-2 )
$$\norm{R^{k}_{n}(t)}_{L^{p}}\leq Cc_k(t)2^{-ks}V_{n}(t)\norm{f_{n}(t)}_{B_{p,r}^{s}}\quad{\rm with }\quad \norm{c_k(t)}_{l^{r}}=1
$$
leads to
\begin{equation*}
\norm{f_n(t)}_{B^{s}_{p,r}}\leq\norm{f_0}_{B^{s}_{p,r}}+\int_{0}^{t}\Big(\norm{g_n(\tau)}_{B^{s}_{p,r}}+CV_{n}(\tau)\norm{f(\tau)}_{B_{p,r}^{s}}\Big){\rm d}\tau,
\end{equation*}
where $V_{n}(t):=\|\nabla u_{n}(t)\|_{ L^{\infty}}$. Applying the Gronwall inequality, we obtain
\begin{equation*}
\norm{f_n(t)}_{B^{s}_{p,r}}\leq e^{C\int_{0}^{t}V_{n}(\tau){\rm d}\tau}\Big(\norm{f_0}_{B^{s}_{p,r}}+\int^{t}_{0}e^{-C\int_{0}^{\tau}V_{n}(s){\rm d}s}\norm{g_n(\tau)}_{B^{s}_{p,r}}{\rm d}\tau\Big).
\end{equation*}
In the following, we shall show that, up to an subsequence, the sequence $(f_n)_{n\in\NN}$ converges in $\mathcal{D}'(\RR_+\times\RR^{n})$ to a solution $f$ of
\eqref{continuce} which has the desired regularity properties. First, one may write
\begin{equation*}
\partial_tf_n-g_n=-u_n\cdot\nabla f_n+\Delta_h f_n.
\end{equation*}
Hence $u_n\in L^{\sigma}([0,T];B_{\infty,\infty}^{-m})$ enables us to conclude that $\partial_tf_n-g_n$ is bounded
in $L^{\sigma}([0,T];B_{p,\infty}^{-M})$ for some sufficiently large $M>0$. For the sake of convenience,
let $$\bar f_n(t):=f_n(t)-\int_0^tg_n(\tau){\rm d}\tau.$$
 Thanks to the imbedding theorem, one may deduce that $(\bar f_n)_{n}$ belongs to $ \mathcal{C}^{\beta}([0,T];B_{p,\infty}^{-M})$ with $\beta>0$ and hence uniformly equicontinuous with value in $B_{p,\infty}^{-M}$. Now, let $(\chi_{l})_{l\in\NN}$ be a sequence of $C^{\infty}_{0}(\RR^{n})$ cut-off functions supported in the ball $B(0,l+1)$ of $\RR^{n}$ and equal to 1 in a neighborhood of $B(0,l)$. On the other hand, by Theorem 2.94 of \cite{BCD11}, we know that the map $u\mapsto \chi_{l}u$ is compact from $B_{p,r}^{s}$ to $B_{p,\infty}^{-M}$. By using Ascoli's theorem and the Cantor diagonal process, there exists  a subsequence which we still denote by $(\bar f_{n})_{n\in\NN}$ such that, for all $l\in\NN$,
\begin{equation*}
\chi_{l}\bar f_{n}\rightarrow_{n\rightarrow\infty}\chi_l \bar f\quad {\rm in }\quad \mathcal{C}([0,T];B_{p,\infty}^{-M}).
\end{equation*}
It follows that the sequence $(\bar f_{n})_{n\in\NN}$ converges to some distribution $\bar f$ in $\mathcal{D}'(\RR_+\times\RR^{n})$.

The only problem is to pass to the limit in $\mathcal{D}'(\RR_+\times\RR^{n})$ for the  convection term. Let $\psi\in C^{\infty}_{0}(\RR_+\times\RR^{n})$ and $l\in\NN$ be such that ${\rm supp}\psi\subset[0,T]\times B(0,l)$, we  have the decomposition
 \begin{equation}\label{decomp}\psi u_{n}\cdot\nabla f_{n}-\psi u\cdot\nabla f=\psi u_{n}\cdot\big(\nabla( \chi_{l}f_{n}-\chi_{l}f)\big)-\psi \chi_{l}(u_n-u)\cdot\nabla f
 \end{equation}Coming back to the uniform estimates of $f_n\in L^{\infty}([0,T];B_{p,r}^{s})$, the Fatou properties of Besov space ensures $\bar f$ belong to $L^{\infty}([0,T];B_{p,r}^{s})$. By preceding argument, we find that $\chi_{l}\bar f_{n}$ tends to $\chi_l \bar f$ in $\mathcal{C}([0,T];B^{s-\varepsilon}_{p,\infty})$ for all $\varepsilon>0$ and $l\in\NN$. Therefore, both two terms in the right of \eqref{decomp} tend to  zero in $L^{\infty}([0,T];B^{s-1-\varepsilon}_{p,\infty})$. On the other hand, the sequences $(f_{0,n})_{n\in\NN},(g_{n})_{n\in\NN}$ and $(u_{n})_{n\in\NN}$ converges to $f_0, g$ and $u$, respectively. So, we finally conclude that $f:=\bar f+\int_0^tg(\tau){\rm d}\tau$ is a solution of \eqref{continuce}.

It  remains to prove that $f\in \mathcal{C}([0,T];B^{s}_{p,r})$, when $r<\infty$. Making use of  uniform estimates of $\bar f_{n}$, one can deduce that $\partial_tf$ belongs to $L^{1}([0,T];B_{p,\infty}^{-M})$. Obviously, for fixed $k$, $\partial_t\Delta_kf$ belongs to $L^{1}([0,T];L^{p})$ so that each $\Delta_kf$ is continuous in time with value  in $L^p$. This implies $S_{k}f\in\mathcal{C}([0,T];B^{s}_{p,r})$ for all $k\in\NN$. Since
$$\Delta_{k'}(f-S_{k}f)=\sum_{|k''-k'|\leq1,k''\geq k}\Delta_{k'}(\Delta_{k''}f),
$$ then we have
$$\norm{f-S_{k}f}_{B_{p,r}^{s}}\leq C\Big(\sum_{k'\geq k-1}2^{k'sr}\norm{\Delta_{k'}f}_{L^{p}}\Big)^{\frac{1}{r}}.
$$ By the same argument as in proof of \eqref{frequency}, one may conclude that
\begin{equation}\norm{\Delta_{k}f(t)}_{L^{p}}\leq\norm{\Delta_{k}f_0}_{L^{p}}+\int_{0}^{t}\norm{\Delta_{k}g(\tau)}_{L^{p}}{\rm d}\tau+C\int_{0}^{t}c_k(t)2^{-ks}V(t)\norm{f(\tau)}_{B_{p,r}^{s}}{\rm d}\tau.
\end{equation}
It follows that
\begin{align*}
\norm{f-S_kf}_{L^{\infty}_{T}(B_{p,r}^{s})}\leq &C\Big(\sum_{k'\geq k-1}\big(2^{k's}\norm{\Delta_{k'}f_{0}}_{L^{p}}\big)^{r}\Big)^{\frac{1}{r}}\\
&+C\int_{0}^{T}\Big(\sum_{k'\geq k-1}\big(2^{k's}\norm{\Delta_{k'}g(\tau)}_{L^{p}}\big)^{r}\Big)^{\frac{1}{r}}{\rm d}\tau\\
&+C\norm{f}_{L^{\infty}_{T}(B_{p,r}^{s})}\int_{0}^{T}\Big(\sum_{k'\geq k-1}c_{k'}^{r}(\tau)\Big)^{\frac{1}{r}}V(\tau){\rm d}\tau
\end{align*}
The fact $f_{0}\in B_{p,r}^{s}$ ensures that the first term tends to zero as $k$ goes to infinity. Since $g,V\in L^{1}_{T}(B_{p,r}^{s})$, one conclude that
the terms in the integrals also tends to zero for almost every $t$. This together with the Lebesgue dominated convergence theorem entails $\norm{f-S_kf}_{L^{\infty}_{T}(B_{p,r}^{s})}$ tends to zero as $k$ goes to infinity. Thus, we can conclude that $f$ belongs to $\mathcal{C}([0,T];B^{s}_{p,r})$.

For the case $r=\infty$, by using the interpolation theorem, we deduce that for any $t_0\in[0,T]$ and $s'\in]-M,s[$, there exists a constant $\theta\in]0,1[$ depending on $s'$ such that
\begin{align*}\norm{u(t)-u(t_0)}_{B_{p,\infty}^{s'}}\leq&\norm{u(t)-u(t_0)}^{\theta}_{B_{p,\infty}^{-M}}\norm{u(t)-u(t_0)}^{1-\theta}_{B_{p,\infty}^{s}}\\
\leq& 2\norm{u(t)-u(t_0)}^{\theta}_{B_{p,\infty}^{-M}}\norm{u}^{1-\theta}_{L^{\infty}_{T}(B_{p,\infty}^{s})}.
\end{align*}
This together with the fact $f\in \mathcal{C}([0,T];B^{-M}_{p,\infty})$ yields $f\in \mathcal{C}([0,T];B^{s'}_{p,\infty})$ for all $s'<s$. Now, we only need to prove that $f\in \mathcal{C}_{w}([0,T];B^{s}_{p,\infty})$. Indeed, for fixed $\phi\in\mathcal{S}(\RR^n)$, the low-high decomposition technique leads to
\begin{align*}
\langle f(t),\phi\rangle=\langle S_{k}f(t),\phi\rangle+\langle({\rm Id}- S_{k})f(t),\phi\rangle=\langle S_{k}f(t),\phi\rangle+\langle f(t),({\rm Id}- S_{k})\phi\rangle.
\end{align*}
Combining this with $f\in \mathcal{C}([0,T];B^{s'}_{p,\infty})$ gives that the function $t\mapsto\langle S_{k}f(t),\phi\rangle$ is continuous. As for the second term, we have
$$|\langle f(t),({\rm Id}- S_{k})\phi\rangle|\leq\norm{f}_{B_{p,\infty}^{s}}\norm{\phi-S_k\phi}_{B_{p',1}^{-s}}.
$$It follows that $\langle f(t),({\rm Id}- S_{j})\phi\rangle$ tends to zero uniformly on $[0,T]$ as $k$ goes to infinity. This means that $f(t)\in \mathcal{C}_{w}([0,T];B^{s}_{p,\infty})$.

Now, we focus on the proof of the uniqueness. Let $f_{1}$ and $f_2$ solve \eqref{continuce} with the same initial datum. If we define $\delta f=f_{1}-f_{2}$, then $\delta f$ solves
$$\partial_t\delta f+u\cdot\nabla\delta f-\Delta_h\delta f=0.
$$ This together with the estimate \eqref{guji} ensures the uniqueness of solution of \eqref{continuce}.
\end{proof}

The last part of the appendix is devoted to the proof of losing a priori estimate for \eqref{eq1.1} with $\nabla u\in L^1([0,T];LL)$, where the LogLip space $LL$ is the set of those functions $f$ which belong to  $\mathcal{S}'~$¡¡and satisfy
\begin{equation}\label{LL}
\norm{f}_{LL}:=\sup_{2\leq q<\infty} \frac{\norm{\nabla S_{q}f}_{L^{\infty}}}{q+1}<\infty.¡¡
\end{equation}
This estimate is the cornerstone to the proof of uniqueness in Theorem \ref{lose-global}. In the sprite of \cite{BCD11,dp2}, we prove linear losing  a priori estimates for the general anisotropic system with convection. More precisely, we have:
\begin{proposition}\label{lostest}
Let $s_{1}\in[-\frac12,1[$ and assume that $s\in]s_{1},1[$. Let $v$ satisfies the following system
\begin{equation}\label{lost}
\begin{cases}
\partial_tv+u\cdot\nabla v-\Delta_hv+\nabla p=f+ge_{3},\\
\text{\rm div}v={\rm div}u=0
\end{cases}
\end{equation}
with initial data $v_0\in H^{s}$ and source terms $f\in L^{1}([0,T];H^{s})$, $g\in L^{2}([0,T];H^{s-1})$. Assume in addition that, for some $h(t)\in L^{1}[0,T]$  satisfying
\begin{equation}\label{trans}
\norm{ u}_{LL}\leq h(t).
\end{equation}
 Then there exists a constant $C$ such that for any $\lambda>C,T>0$ and
 $$s_{t}:=s-\lambda\int_{0}^{t}h(\tau){\rm d}\tau,
 $$
the following estimate holds
 \begin{align*}
\norm{v(t)}_{H^{s_{t}}}+\norm{\nabla_hv}_{L^{2}_{t}H^{s_{t}}}
\leq C(1+\sqrt{t})\exp\Big({\frac{C}{\lambda}\int_0^t h(\tau){\rm d}\tau}\Big)\big(\norm{\rho_0}_{H^{s}}+\norm{f}_{L^1_{t}H^{s}}
+\norm{g}_{L^2_{t}H^{s-1}}\big).
\end{align*}
\end{proposition}
\begin{proof}
Applying the operator $\Delta_{q}$ to the system \eqref{lost}, we find that for all $q\geq-1$, the $f_{q}$ solves the following equations
\begin{equation*}
\partial_tv_q+ S_{q-1}u\cdot v_q-\Delta_hv_q+\nabla p_q=f_q+g_qe_3+F_q(u,v)
\end{equation*}
with $F_q(u,v)= S_{q-1}u\cdot\nabla v_q-\Delta_q(u\cdot\nabla v).$

Taking the $L^2$-inner product to the above equation with $v_q$ and using $\text{\rm div}u=0$, we see that
\begin{equation}\label{lost-1}
\frac12\frac{\rm d}{{\rm d}t}\norm{v_q}^{2}_{L^{2}}+\norm{\nabla_hv_q}^{2}_{L^{2}}=\int f_qv_q {\rm d}x+\int g_qv^{3}_q {\rm d}x+\int F_q(u,v)v_q {\rm d}x.
\end{equation}
Assume that $q\geq 0$.
Applying the Bernstein and the Young inequalities, we can deduce that
\begin{align*}
\int g_qv^{3}_q {\rm d}x\leq &C2^{-q}\norm{g_{q}}_{L^{2}}\norm{\nabla v^{3}_{q}}_{L^{2}}\leq\frac14\norm{\nabla v^{3}_{q}}^{2}_{L^{2}}+C2^{-2q}\norm{g_{q}}^{2}_{L^{2}}\\
\leq& \frac14\norm{\nabla_{h} v^{3}_{q}}^{2}_{L^{2}}+\frac14\norm{\partial_{3} v^{3}_{q}}^{2}_{L^{2}}+C2^{-2q}\norm{g_{q}}^{2}_{L^{2}}\\
\leq&\frac12\norm{\nabla_{h} v_{q}}^{2}_{L^{2}}+C2^{-2q}\norm{g_{q}}^{2}_{L^{2}},
\end{align*}
in the last line we have used the fact $\text{div} v=0.$

Integrating the both sides of \eqref{lost-1} with respect to $t$, we get for all $q\geq 0$,
\begin{equation*}
\norm{v_{q}}^{2}_{L^{\infty}_{t}L^{2}}+\norm{\nabla_hv_{q}}^{2}_{L^{2}_{t}L^{2}}
\leq\norm{v_{q}(0)}^{2}_{L^{2}}+2\norm{f_{q}}^{2}_{L^{1}_{t}L^{2}}+C2^{-2q}\norm{g_{q}}^{2}_{L^{2}_{t}L^{2}}
+2\norm{F_{q}(u,v)}^{2}_{L^{1}_{t}L^{2}}.
\end{equation*}
For $q=-1$, we merely have
\begin{equation*}
\norm{v_{-1}(t)}_{L^{2}}\leq \norm{v_{-1}(0)}_{L^{2}}+\int_{0}^{t}\big(\norm{f_{-1}(\tau)}_{L^{2}}+\norm{g_{-1}(\tau)}_{L^{2}}+\norm{F_{-1}(u,v)(\tau)}_{L^{2}}\big){\rm d}\tau.
\end{equation*}
On the other hand, by the Bernstein inequality, we know that
$$\norm{\nabla_hv_{-1}}_{L^{2}_{t}L^{2}}\leq Ct^{\frac{1}{2}}\norm{v_{-1}}_{L^{\infty}_{t}L^{2}}.$$
Therefore, for all $q\geq-1$, we have
\begin{equation}\label{equ-lose}
\begin{split}
&\norm{v_q}_{L^{\infty}_{t}L^{2}}+\norm{\nabla_hv_q}_{L^{2}_{t}L^{2}}\\
\leq&2(1+\sqrt{t})\Big(\norm{v_q(0)}_{L^{2}}+\norm{f_q}_{L^{1}_{t}L^{2}}+C2^{-q}\norm{g_q}_{L^{2}_{t}L^{2}}+\norm{F_q(u,v)}_{L^{1}_{t}L^{2}}\Big).
\end{split}
\end{equation}
From a standard commutator estimate \mbox{(see e.g. \cite{BCD11}, Chap. 2)}, we know that for all $\varepsilon\in]0,\frac{s+1}{2}[,q\geq -1$ and $t\in [0,T]$
\begin{equation}\label{tr-2}
2^{q(s-\varepsilon)}\norm{F_q(u,v)(t)}_{L^2}\leq Cc_{q}{(2+q)}h(t)\norm{v(t)}_{H^{s-\varepsilon}}\quad \text{with } c_{q}\in l^{2}
\mathcal{}\end{equation}
for some constant $C$ depending only on $s.$

Set $s_{t}:=s-\lambda\int_{0}^{t}h(\tau){\rm d}\tau$ for $t\in[0,T]$. Collecting \eqref{equ-lose}  and \eqref{tr-2}
yields that
\begin{equation}\label{tr-3}
\begin{split}
2^{(2+q)s_{t}}\norm{v_{q}}_{L^{2}}\leq 2(1+\sqrt{t})&\Big( 2^{(2+q)s}\norm{v_q(0)}_{L^{2}}2^{-\eta(2+q)\int_{0}^{t}h(\tau){\rm d}\tau}\\&+\int_{0}^{t}2^{(2+q)s_{\tau}}\norm{f_q(\tau)}_{L^{2}}2^{-\eta(2+q)\int_{\tau}^{t}h(\tau'){\rm d}\tau'}{\rm d}\tau\\
&+\Big(\int_{0}^{t}2^{2(2+q)s_{\tau}}2^{-2q}\norm{g_q(\tau)}^{2}_{L^{2}}2^{-2\lambda(2+q)\int_{\tau}^{t}h(\tau'){\rm d}\tau'}{\rm d}\tau\Big)^{\frac{1}{2}}\\
&+Cc_{q}(2+q)\int_{0}^{t}h(\tau)2^{-\lambda(2+q)\int_{\tau}^{t}h(\tau'){\rm d}\tau'}\norm{f(\tau)}_{H^{s_{\tau}}}{\rm d}\tau\Big).
\end{split}
\end{equation}
For the last term of \eqref{tr-3}, we observe that
\begin{align*}
&Cc_{q}(2+q)\int_{0}^{t}h(\tau)2^{-\lambda(2+q)\int_{\tau}^{t}h(\tau'){\rm d}\tau'}\norm{f(\tau)}_{H^{s_{\tau}}}{\rm d}\tau\\
\leq&Cc_{q}\frac{1}{\lambda\log2}\int_{0}^{t}{\rm d}2^{-\lambda(2+q)\int_{\tau}^{t}h(\tau'){\rm d}\tau'}\sup_{\tau\in[0,t]}\norm{f(\tau)}_{H^{s_{\tau}}}\\
=&c_{q}\frac{C}{\lambda\log2}\big(1-2^{-\lambda(2+q)\int_{0}^{t}h(\tau){\rm d}\tau}\big)\sup_{\tau\in[0,t]}\norm{f(\tau)}_{H^{s_{\tau}}}.
\end{align*}
Thus, multiplying  $2^{qs_{\tau}}$ and taking the $l^{2}$-norm of both sides of \eqref{tr-3} over $q\geq-1$, we get
\begin{align*}
\sup_{\tau\in[0,t]}\norm{f(\tau)}_{H^{s_{\tau}}}\leq 2(1+\sqrt{t})&\Big(\norm{f_{0}}_{H^{s}}+\norm{f(\tau)}_{L^{1}_{t}H^{s_{\tau}}}+\norm{g(\tau)}_{L^{2}_{t}H^{s_{\tau}-1}}\\
&+\frac{C}{\lambda\log2}\sup_{\tau\in[0,t]}\norm{f(\tau)}_{H^{s_{\tau}}}\Big).
\end{align*}
Choosing  $\lambda_0$ such that $\frac{2C(1+\sqrt{t})}{\lambda_0\log2}=\frac12$, we get by the Gronwall inequality that for any $\lambda>\lambda_0$,
\begin{equation*}
\sup_{t\in[0,t]}\norm{f(t)}_{H^{s_{\tau}}}\leq 2(1+\sqrt{t})e^{\frac{C}{\lambda}\int_0^t h(\tau){\rm d}\tau}\Big(\norm{f_{0}}_{H^{s}}+\norm{f(\tau)}_{L^{1}_{t}H^{s_{\tau}}}+\norm{g(\tau)}_{L^{2}_{t}H^{s_{\tau}-1}}\Big).
\end{equation*}
This implies the desired result.
\end{proof}
\subsection*{Acknowledgements.}The authors are partly supported by the NSF
of China  No. 11171033. The authors would like to thank Dr.L. Xue
 for his valuable discussion.

\end{document}